\documentclass[11pt]{amsart}

\usepackage{amsmath,amssymb,amscd,amsthm,amsxtra}

\usepackage{graphicx, mathrsfs}
\usepackage{amssymb,amsmath,amsthm}
\usepackage{mathrsfs,dsfont}

\makeatletter
\DeclareRobustCommand\widecheck[1]{{\mathpalette\@widecheck{#1}}}
\def\@widecheck#1#2{%
   \box\z@\hbox{\m@th$#1#2$}%
   \box\tw@\hbox{\m@th$#1%
      \widehat{%
         \vrule\@width\z@\@height\ht\z@
         \vrule\@height\z@\@width\wd\z@}$}%
   \dp\tw@-\ht\z@
   \@tempdima\ht\z@ \advance\@tempdima2\ht\tw@ \divide\@tempdima\thr@@
   \box\tw@\hbox{%
      \raise\@tempdima\hbox{\scalebox{1}[-1]{\lower\@tempdima\box\tw@}}}%
   {\ooalign{\box\tw@ \cr \box\z@}}}
\makeatother

\newtheorem{theorem}{Theorem} [section]

\newtheorem{lemma}[theorem]{Lemma}
\newtheorem{proposition}[theorem]{Proposition}
\newtheorem{remark}[theorem]{Remark}
\newtheorem{example}{Example}

\newtheorem{definition}[theorem]{Definition}
\newtheorem{corollary}[theorem]{Corollary}

\usepackage{tikz}
\usetikzlibrary{calc}
\tikzset{
    vertex/.style = {
        outer sep = 0pt,
        inner sep = 0pt,
    }
}
\tikzstyle{every node}=[circle, draw, 
                        inner sep=0pt, minimum size=0.2cm]

\begin{document}
\title[A rigorous derivation of the defocusing cubic NLS on $\mathbb{T}^3$]{A rigorous derivation of the defocusing cubic nonlinear Schr\"{o}dinger equation on $\mathbb{T}^3$ from the dynamics of many-body quantum systems}
\author[Vedran Sohinger]{Vedran Sohinger}
\address{
University of Pennsylvania, Department of Mathematics, David Rittenhouse Lab, Office 3N4B, 209 South 33rd Street, Philadelphia, PA 19104-6395, USA}
\email{vedranso@math.upenn.edu}
\urladdr{http://www.math.upenn.edu/~vedranso/}
\keywords{Gross-Pitaevskii hierarchy, Nonlinear Schr\"{o}dinger equation, density matrices, collision operator, Duhamel iteration, Quantum de Finetti Theorem, Trace class operators, A rigorous derivation of the NLS, Many-body quantum systems, Propagation of Chaos}
\subjclass[2010]{35Q55, 70E55}
\thanks{V. S. was supported by a Simons Postdoctoral Fellowship.}
\maketitle

\begin{abstract}
In this paper, we will obtain a rigorous derivation of the defocusing cubic nonlinear Schr\"{o}dinger equation on the three-dimensional torus $\mathbb{T}^3$ from the many-body limit of interacting bosonic systems. This type of result was previously obtained on $\mathbb{R}^3$ in the work of Erd\H{o}s, Schlein, and Yau \cite{ESY2,ESY3,ESY4,ESY5}, and on $\mathbb{T}^2$ and $\mathbb{R}^2$ in the work of Kirkpatrick, Schlein, and Staffilani \cite{KSS}. Our proof relies on an unconditional uniqueness result for the Gross-Pitaevskii hierarchy at the level of regularity $\alpha=1$, which is proved by using a modification of the techniques from the work of T. Chen, Hainzl, Pavlovi\'{c} and Seiringer \cite{ChHaPavSei} to the periodic setting. These techniques are based on the Quantum de Finetti theorem in the formulation of Ammari and Nier \cite{AmmariNier1,AmmariNier2} and Lewin, Nam, and Rougerie \cite{LewinNamRougerie}. In order to apply this approach in the periodic setting, we need to recall multilinear estimates obtained by Herr, Tataru, and Tzvetkov \cite{HTT}.

Having proved the unconditional uniqueness result at the level of regularity $\alpha=1$, we will apply it in order to finish the derivation of the defocusing cubic nonlinear Schr\"{o}dinger equation on $\mathbb{T}^3$, which was started in the work of Elgart, Erd\H{o}s, Schlein, and Yau \cite{EESY}. In the latter work, the authors obtain all the steps of Spohn's strategy for the derivation of the NLS \cite{Spohn}, except for the final step of uniqueness. Additional arguments are necessary to show that the objects constructed in \cite{EESY} satisfy the assumptions of the unconditional uniqueness theorem. Once we achieve this, we are able to prove the derivation result. In particular, we show \emph{Propagation of Chaos} for the defocusing Gross-Pitaevskii hierarchy on $\mathbb{T}^3$ for suitably chosen initial data.
\end{abstract}


\section{Introduction}
\subsection{Setup of the problem}

In this paper, we will consider the \emph{Gross-Pitaevskii hierarchy} on the three-dimensional torus $\mathbb{T}^3$. Given a more general spatial domain $\Lambda=\mathbb{R}^d$ or $\mathbb{T}^d$, the Gross-Pitaevskii (GP) hierarchy on $\Lambda$ is a system of infinitely many equations given by:
\begin{equation}
\label{GrossPitaevskiiHierarchy}
\begin{cases}
i \partial_t \gamma^{(k)} + (\Delta_{\vec{x}_k}-\Delta_{\vec{x}_k'}) \gamma^{(k)}= \lambda \cdot \sum_{j=1}^{k} B_{j,k+1} (\gamma^{(k+1)})\\
\gamma^{(k)}\big|_{t=0}=\gamma_0^{(k)}.
\end{cases}
\end{equation}

For $k \in \mathbb{N}$, $\gamma_0^{(k)}$ is a complex-valued function on $\Lambda^k \times \Lambda^k$. Such a function is referred to as a \emph{density matrix of order $k$}. Each $\gamma^{(k)}=\gamma^{(k)}(t)$ is a time-dependent density matrix of order $k$. Here, $\Delta_{\vec{x}_k}$ and $\Delta_{\vec{x}'_k}$ denote the Laplace operators in the first and second component of $\Lambda^k$. More precisely, $\Delta_{\vec{x}_k}:=\sum_{j=1}^{k} \Delta_{x_j}$ and $\Delta_{\vec{x}'_k}:=\sum_{j=1}^{k}\Delta_{x'_j}$. Furthermore, $B_{j,k+1}$ denotes the \emph{collision operator} given by: 
\begin{equation}
\notag
B_{j,k+1}(\sigma^{(k+1)}):=Tr_{k+1} \big[\delta(x_j-x_{k+1}),\sigma^{(k+1)}\big],
\end{equation} 
for $\sigma^{(k+1)}$ a density matrix of order $k+1$. $Tr_{k+1}$ denotes the trace in the variable $x_{k+1}$. In particular $B_{j,k+1}(\sigma^{(k+1)})$ is a density matrix of order $k$. A more precise definition is given in \eqref{Collision_Operator} below. Finally $\lambda \in \mathbb{R}$ is a non-zero coupling constant. In this paper, we will take $\lambda \in \{1,-1\}$. If $\lambda=1$, we will call the GP hierarchy \emph{defocusing}, and if $\lambda=-1$, we will call the hierarchy \emph{focusing}.

The Gross-Pitaevskii hierarchy \eqref{GrossPitaevskiiHierarchy} is closely related to the cubic nonlinear Schr\"{o}dinger equation (NLS) on $\Lambda$. We recall that the cubic NLS on $\Lambda$ is given by:
\begin{equation}
\label{NLS}
\begin{cases}
i \partial_t u + \Delta u = \lambda \cdot |u|^2  u,\,\,\mbox{on}\,\,\mathbb{R}_t \times \Lambda\\
u\,\big|_{t=0}=\phi,\,\mbox{on}\,\Lambda.
\end{cases}
\end{equation}
This problem is called \emph{defocusing} if $\lambda=1$ and \emph{focusing} if $\lambda=-1$. 
Given a solution $u$ of \eqref{NLS}, we can build a solution to \eqref{GrossPitaevskiiHierarchy} by taking tensor products. In other words, 
\begin{equation}
\label{factorized_solution}
\gamma^{(k)}(t,\vec{x}_k;\vec{x}_k'):=\prod_{j=1}^k u(t,x_j) \overline{u(t,x_j')}=|u \rangle \langle u|^{\otimes k}(t,\vec{x}_k;\vec{x}_k')
\end{equation}
solves \eqref{GrossPitaevskiiHierarchy} for initial data given by $\gamma_0^{(k)}=|\phi\rangle \langle \phi|^{\otimes k}$. 
Here, we denote by $| \cdot \rangle \langle \cdot|$ the \emph{Dirac bracket}, which is defined as $|f \rangle \langle g| (x,x'):=f(x) \overline{g(x')}$. These are called the \emph{factorized solutions}. In this way, we can embed the nonlinear problem \eqref{NLS} into the linear problem \eqref{GrossPitaevskiiHierarchy}.

In addition to the structural connection noted above, the study of the GP hierarchy is important in the context of the rigorous derivation of the NLS from the dynamics of many-body quantum systems. In particular, given a potential $V:\Lambda \rightarrow \mathbb{R}$, we can consider the $N$-body Hamiltonian $H_N$:
\begin{equation}
\notag
H_N:=-\sum_{j=1}^{N}\Delta_j+\frac{1}{N}\sum_{\ell < j}^{N}V_N(x_{\ell}-x_j)
\end{equation}
defined on a dense subspace of $L^2_{sym}(\Lambda^N)$, the space of all permutation-symmetric elements of $L^2(\Lambda^N)$. Here, $V_N$ is a rescaled version of $V$ involving $N$. Given $\Psi_{N,0} \in L^2_{sym}(\Lambda^N)$, we can study the $N$-body Schr\"{o}dinger equation associated to $H_N$ with initial data $\Psi_{N,0}$:
\begin{equation}
\label{PsiNt_equation}
\begin{cases}
i \partial_t \psi_{N,t}=H_N \,\psi_{N,t}\\
\Psi_{N,t}\big|_{t=0}=\Psi_{N,0}.
\end{cases}
\end{equation}
The solution $\psi_{N,t}$ belongs to $L^2_{sym}(\Lambda^N)$. In particular, $\|\Psi_{N,t}\|_{L^2(\Lambda^N)}=\|\Psi_{N,0}\|_{L^2(\Lambda^N)}.$
We define:
\begin{equation}
\label{gammakNt_original}
\gamma^{(k)}_{N,t}:=Tr_{k+1,\ldots,N} \,\big|\Psi_{N,t} \rangle \langle \Psi_{N,t}\big|.
\end{equation}
Here $Tr_{k+1,\ldots,N}$ denotes the partial trace in $x_{k+1},\ldots,x_N$. By definition, if $k>N$, we take $\gamma^{(k)}_{N,t}:=0$. This notation is explained in more detail in Subsection \ref{Trace class} below.

The sequence $(\gamma^{(k)}_{N,t})_k$ then solves the  Bogoliubov-Born-Green-Kirkwood-Yvon (BBGKY) hierarchy:

\begin{equation}
\label{BBGKY}
i \partial_t \gamma^{(k)}_{N,t} + \big(\Delta_{\vec{x}_k}-\Delta_{\vec{x}'_k}\big) \gamma^{(k)}_{N,t}=
\end{equation}
\begin{equation}
\notag
\frac{1}{N} \sum_{\ell<j}^{k} \big[V_{N}(x_{\ell}-x_j),\gamma^{(k)}_{N,t}\big] + \frac{N-k}{N}\sum_{j=1}^{k} Tr_{k+1} \big[V_{N}(x_j-x_{k+1}),\gamma^{(k+1)}_{N,t}\big].
\end{equation}



We note that, formally, the BBGKY hierarchy converges to the defocusing GP hierarchy as $N \rightarrow \infty$. This heuristic can be made formal in the following sense:
given $\phi \in L^2(\Lambda)$ with $\|\phi\|_{L^2(\Lambda)}=1$, under additional assumptions on the sequence of initial data $(\Psi_{N})_N \in \bigoplus_{N \in \mathbb{N}} L^2_{sym}(\Lambda^N)$ in terms of $\phi$, one wants to show that there exists a sequence $N_j \rightarrow \infty$, which is independent of $k \in \mathbb{N}$ and $t\in [0,T]$ such that:
\begin{equation}
\label{convergence}
Tr\,\Big|\gamma^{(k)}_{N_j,t}-|S_t(\phi) \rangle \langle S_t(\phi)|^{\otimes k}\Big| \rightarrow 0,
\end{equation}
as $j \rightarrow \infty$. Here, $Tr\,$ denotes the trace and  $S_t$ denotes the flow map for \eqref{NLS}. 

We will refer to \eqref{convergence} as a \emph{rigorous derivation of the cubic NLS equation on $\Lambda$ from the dynamics of many-body quantum systems}. This type of result is also referred to as \emph{Propagation of Chaos} in the sense that the particles become decoupled in the limit. In our paper, we will study this derivation in the context of the defocusing problem.

A strategy for proving \eqref{convergence} was developed by Spohn \cite{Spohn}, and it consists of first showing that the sequence $(\gamma^{(k)}_{N,t})_N$ is relatively compact and that the limit points solve the GP hierarchy. Once this is established, one shows the GP hierarchy admits unique solutions in the class of objects obtained in the limit. This is a non-trivial step, due to the fact that the GP hierarchy is an infinite and non-closed system of PDEs.

Spohn \cite{Spohn} applied this approach in order to derive the nonlinear Hartree equation $i\partial_t u + \Delta u = (V*|u|^2) \cdot u$ on $\mathbb{R}^3$ when $V \in L^{\infty}(\mathbb{R}^3)$. The case of a Coulomb potential $V(x)=\pm\frac{1}{|x|}$ was later solved by Bardos, Golse, and Mauser \cite{BGM} and Erd\H{o}s and Yau \cite{EY}. An alternative proof of the latter result was subsequently given by Fr\"{o}hlich, Knowles, and Schwarz. In a series of works \cite{ESY2,ESY3,ESY4,ESY5}, Erd\H{o}s, Schlein, and Yau applied this strategy to obtain a derivation of the defocusing cubic nonlinear Schr\"{o}dinger equation on $\mathbb{R}^3$. The uniqueness step required a use of Feynman graph expansions \cite[Sections 9 and 10]{ESY2}. Subsequently, a combinatorial reformulation of the uniqueness proof on $\mathbb{R}^3$ was given by Klainerman and Machedon \cite{KM}. This argument is applicable in a slightly different class of density matrices. The authors prove the uniqueness result in this class by using spacetime estimates reminiscent of their earlier work on null-forms for the nonlinear wave equation \cite{KM2}. The fact that the objects obtained in the limit in the procedure outlined above satisfy the assumptions needed to apply the boardgame argument was first verified by Kirkpatrick, Schlein, and Staffilani \cite{KSS} when $\Lambda=\mathbb{R}^2$ and $\Lambda=\mathbb{T}^2$. Related results were subsequently proven by T. Chen and Pavlovi\'{c} \cite{CP} and X. Chen and Holmer \cite{ChenHolmer2} when $\Lambda=\mathbb{R}^3$. 

In the author's joint work with Gressman and Staffilani \cite[Proposition 3.3]{GrSoSt}, it was noted that the spacetime estimate for the free evolution operator needed to apply the boardgame argument as in \cite{KM} does not hold in the energy space when $\Lambda=\mathbb{T}^3$. This is is in sharp contrast to what happens when $\Lambda=\mathbb{R}^3$ or $\Lambda=\mathbb{T}^2$. Heuristically, this is a manifestation of the weaker dispersion on periodic domains, which also becomes weaker as the dimension becomes larger. In particular, it is not possible to apply the spacetime estimate and show uniqueness of solutions in the class from \cite{KM} in regularity $\alpha=1$. A conditional uniqueness result in this class when $\alpha=1$ is still a relevant open problem, which will require new tools to solve.

	Recently, a new approach to studying the uniqueness problem was taken by T. Chen, Hainzl, Pavlovi\'{c}, and Seiringer in \cite{ChHaPavSei}. Here, the authors prove an unconditional uniqueness result for the GP hierarchy on $\mathbb{R}^3$ by means of the \emph{Quantum de Finetti Theorem}. This theorem is a quantum analogue of the theorem of de Finetti on exchangeable sequences of random variables \cite{deFinetti1,deFinetti2}. Related results to those of de Finetti were subsequently proven by Dynkin \cite{Dynkin}, Hewitt and Savage \cite{HewittSavage}, and Diaconis and Freedman \cite{DiaconisFreedman}. The Quantum de Finetti theorem states that, under certain assumptions, density matrices can be viewed as averages over factorized states. The precise statement is recalled in Theorem \ref{Weak Quantum de Finetti} below. Similar results were first proven in the $C^*$ algebra context in the work of Hudson and Moody \cite{HudsonMoody}, and St\o rmer \cite{Stormer}. Connections to density matrices were subsequently made in the work of Ammari and Nier \cite{AmmariNier1,AmmariNier2}, and Lewin, Nam, and Rougerie \cite{LewinNamRougerie}.

	We recall from \cite[Proposition 5.3]{GrSoSt} that it is possible to estimate factorized objects at the level of regularity of the energy space, i.e. when the Sobolev exponent is $\alpha=1$. Thus, it would be reasonable to expect that it is also possible to estimate averages of factorized states, as one obtains from the Quantum de Finetti Theorem. In particular, the first result that we prove is:

\vspace{5mm}

\emph{\textbf{Theorem 1} (Unconditional uniqueness for the GP hierarchy on $\mathbb{T}^3$ when $\alpha=1$)}

\emph{Let us fix $T>0$. Suppose that $(\gamma^{(k)}(t))_k$ is a mild solution to the Gross-Pitaevskii hierarchy in $L^{\infty}_{t \in [0,T]} \mathfrak{H}^1$ such that, for each $t \in [0,T]$, there exist $\Gamma_{N,t} \in L^2_{sym}(\mathbb{T}^{3N} \times \mathbb{T}^{3N})$ which are non-negative as operators with trace equal to $1$, such that:
\begin{equation}
\notag
Tr_{k+1,\ldots,N}\, \Gamma_{N,t} \rightharpoonup^{*} \gamma^{(k)}(t)
\end{equation}
as $N \rightarrow \infty$ in the weak-$*$ topology of the trace class on $L^2_{sym}(\mathbb{T}^{3k})$. Then, the solution $(\gamma^{(k)}(t))_k$ is uniquely determined by the initial data $(\gamma_0^{(k)})_k$.}

\vspace{5mm}

Theorem 1 is stated as Theorem \ref{Unconditional_uniqueness} in Section \ref{An unconditional uniqueness result}.
Given a Sobolev exponent $\alpha \in \mathbb{R}$, the space $\mathfrak{H}^{\alpha}$ is defined in Definition \ref{mathfrakHalpha} below. Furthermore, the concept of a mild solution to the Gross-Pitaevskii hierarchy is given in Definition \ref{Mild_solution} below. We note that Theorem 1 applies both in the \emph{defocusing} ($\lambda=1$) and the \emph{focusing} ($\lambda=-1$) context.
Let us remark that Theorem 1 does not immediately improve on the main result of \cite{GrSoSt} because the spaces which are used in \cite{GrSoSt} are different and cannot in general be compared to the ones used in the present paper. In particular, the norms used in \cite{GrSoSt} are of Hilbert-Schmidt type, whereas the norms used in the present paper are obtained from the trace as in \cite{ESY1,ESY2,ESY3,ESY4,ESY5}. In \cite{GrSoSt}, there is an additional condition involving the collision operator as was the case in the analysis of \cite{KM}.

Let us note that in \cite{EESY}, Elgart, Erd\H{o}s, Schlein and Yau complete the first step in the strategy developed by Spohn in the context of the derivation of the defocusing cubic NLS on $\mathbb{T}^3$. In other words, they show that the sequence $(\gamma^{(k)}_{N,t})_N$ is relatively compact and that the limit points of this sequence solve the Gross-Pitaevskii hierarchy. Their analysis applies at the level of regularity $\alpha=1$.

In \cite{EESY}, the authors state the uniqueness step as an open problem. In our paper, we will resolve this problem. As a result, we will obtain a derivation of the defocusing cubic NLS on $\mathbb{T}^3$. The result that we prove is:

\vspace{5mm}

\emph{\textbf{Theorem 2} (A derivation of the defocusing cubic NLS on $\mathbb{T}^3$)}

\emph{Let $\phi \in H^1(\mathbb{T}^3)$ with $\|\phi\|_{L^2(\mathbb{T}^3)}=1$ be given. Suppose that the sequence of initial data in the $N$-body Schr\"{o}dinger equation $(\Psi_N)_N \in \bigoplus_{N \in \mathbb{N}} L^2_{sym}(\mathbb{T}^{3N})$ satisfies the assumptions of:
\begin{itemize}
\item[1)] Bounded energy per particle: $\sup_{N \in \mathbb{N}} \langle H_N \Psi_N, \Psi_N \rangle_{L^2(\mathbb{T}^{3N})}<\infty.$ 
\item[2)] Asymptotic factorization: $Tr\,\Big|\, Tr_{\,2,3,\ldots,N} \,|\Psi_N \rangle \langle \Psi_N| - |\phi \rangle \langle \phi| \Big| \rightarrow 0$ as $N \rightarrow \infty.$ 
\end{itemize}
For fixed $N \in \mathbb{N}$, let $(\gamma^{(k)}_{N,t})_k$ denote the solution of the BBGKY hierarchy evolving from initial data $(\gamma^{(k)}_N)_k:=Tr_{k+1,\ldots,N} |\Psi_N \rangle \langle \Psi_N|$. Let $T>0$ be fixed. Then, there exists a sequence $N_j \rightarrow \infty$, such that for all $k \in \mathbb{N}$ and for all $t \in [0,T]$ the convergence \eqref{convergence} holds.}

\vspace{5mm}

Theorem 2 is given as Theorem \ref{Cubic_NLS_T3_derivation} in Section \ref{derivation of defocusing cubic NLS on T3}.
We note that this result is stated only in the defocusing context. In the proof of Theorem 2, we will use a uniqueness argument based on Theorem 1 and the analysis of \cite{EESY}. For details on the application of the uniqueness result in the proof of Theorem 2, we refer the reader to \eqref{Gamma_tilde_infty} and to Remark \ref{place_where_we_use_uniqueness} below. The assumptions of Theorem 2 are satisfied for purely factorized states $\Psi_N=|\phi \rangle \langle \phi|^{\otimes N}$, and the result holds in this case as is noted in Corollary \ref{Evolution_of_purely_factorized_states} below.

\subsection{Physical interpretation}

The Gross-Pitaevskii hierarchy and the nonlinear Schr\"{o}dinger equation have a related physical interpretation. Both objects occur in the context of Bose-Einstein condensation, a state of matter consisting of dilute bosonic particles which are cooled to a temperature close to absolute zero. At such a temperature, these particles tend to occupy the lowest quantum state, which can be expressed mathematically as the ground state of an energy functional related to the NLS. In this context, the NLS equation is sometimes referred to as the 
\emph{Gross-Pitaevskii equation}, following the work of Gross \cite{Gross} and Pitaevskii \cite{Pitaevskii}.

 The physical phenomenon of Bose-Einstein condensation was theoretically  predicted in the pioneering works of Bose \cite{Bose} and Einstein \cite{Einstein} in 1924-1925. Their prediction was verified by experiments conducted independently by the teams led by Cornell and Wieman \cite{CW} and Ketterle \cite{Ket} in 1995. These two groups were awarded the Physics Nobel Prize in 2001.

\subsection{Additional related results}

In addition to the references mentioned above, there is a rich literature devoted to the connection between NLS-type equations and hierarchies similar to \eqref{GrossPitaevskiiHierarchy}. In addition to the strategy developed by Spohn \cite{Spohn}, outlined above, an independent strategy based on Fock space methods was simultaneously developed by Hepp \cite{Hepp} and Ginibre and Velo \cite{GV1,GV2}. Spohn's strategy was subsequently applied in the derivation problem in \cite{ABGT,AGT,BGM,BePoSc1,BePoSc2,BdOS,CP,CT,XC3,XC4,ChenHolmer1,ChenHolmer2,ChenHolmer3,ES,FGS,FL,MichelangeliSchlein,Xie}, whereas the Fock space techniques were subsequently applied in \cite{XC1,XC2,FKP,FKS,GM,GMM1,GMM2}. The question of deriving the nonlinear Hartree equation was revisited in the work of Fr\"{o}hlich, Tsai, and Yau \cite{FrTsYau1,FrTsYau2,FrTsYau3}. Once one obtains a derivation of the NLS-type equation, it is natural to ask what is the rate of convergence. This question was first addressed by Rodnianski and Schlein \cite{RodnianskiSchlein}. Subsequent results
on this problem have been obtained in \cite{Anapolitanos, BdOS,ChenLeeSchlein,XC1,XC2,ErdosSchlein,FKP,GM,GMM1,GMM2,KP,Lee,Luhrmann,MichelangeliSchlein,Pickl1,Pickl2}.
The Gross-Pitaevskii hierarchy has been studied at the $N$-body level by Lieb and Seiringer \cite{LS}, Lieb, Seiringer and Yngvason \cite{LSY,LSY2}, and Lieb, Seiringer, Yngvason, and Solovej \cite{LSSY2}. In these works, the assumption of asymptotic factorization given in the assumption of Theorem 2 was rigorously verified for a sequence of appropriate ground states. For more details on this aspect of the problem, we refer the reader to the expository work \cite{LSSY} and to the references therein. A connection of the above problems with optical lattice models was explored in the work of Aizenman, Lieb, Seiringer, Yngvason, and Solovej \cite{ALSSY1,ALSSY2}. An expository survey of many of the above results can be found in \cite{Schlein}.

The Gross-Pitaevskii hierarchy has been studied as a Cauchy problem in its own right in the recent works of T. Chen and Pavlovi\'{c} \cite{CP1,CP_survey_article,CP4,CP3,CP}, and in their joint works with Tzirakis \cite{CPT1,CPT2}, as well as in the subsequent work of Z. Chen and Liu \cite{CL}. 
The motivation is to study the Gross-Pitaevskii hierarchy as a generalization of the nonlinear Schr\"{o}dinger equation via the factorized solutions and to prove analogues of the known results for the Cauchy problem for the NLS in the context of the GP hierarchy. By appropriately modifying the collision operator, it is also possible to consider a hierarchy which is related to the quintic NLS \cite{CP1,CP_survey_article,CP2,CP4,CP3,CPT1,CPT2}, as well as the NLS with more general power-type nonlinearities \cite{Xie}. 
The Cauchy problem associated to the Hartree equation for infinitely many particles has recently been studied by Lewin and Sabin \cite{LewinSabin1,LewinSabin2}. 
Randomization techniques similar to those used in the context of the Cauchy problem for nonlinear dispersive equations in the work of Bourgain \cite{B2} and Burq and Tzvetkov \cite{BT1} were used in order to study randomized forms of the Gross-Pitaevskii hierarchy in the author's joint work with Staffilani \cite{SoSt}, as well as in the author's work \cite{VS}.

Techniques similar to those used in \cite{ChHaPavSei} have been applied in order to show scattering results in the subsequent work of T. Chen, Hainzl, Pavlovi\'{c} and Seiringer \cite{ChHaPavSei2}. The unconditional uniqueness result of \cite{ChHaPavSei} was subsequently extended to lower regularities by Hong, Taliaferro, and Xie in \cite{HTX}. We note that the methods used in \cite{HTX} do not directly apply to the periodic setting due to the weaker dispersion. It is an interesting open problem to see if the unconditional uniqueness in the periodic section can be extended to lower regularities.

\subsection{Main ideas of the proof}
The proof of the unconditional uniqueness result in Theorem 1 will be based on the Weak Quantum de Finetti Theorem used in the work of T. Chen, Pavlovi\'{c}, Hainzl, and Seiringer \cite{ChHaPavSei} in the non-periodic setting. Several modifications will be necessary in order to apply these methods in the periodic setting. The main point is that one needs to use Strichartz estimates on $\mathbb{T}^3$. We will use the trilinear estimate from the work of Herr, Tataru, and Tzvetkov \cite{HTT}. This result is recalled in Proposition \ref{Star} below. In particular, using this estimate we are able to prove a product estimate at the level of $H^1$ regularity in Proposition \ref{Multilinear_estimate}. Following the terminology of \cite{ChHaPavSei}, this result is used in order to prove the \emph{bound on the $L^2$ level} \eqref{Bound_on_the_L2_level} and the \emph{bound on the $H^1$ level} \eqref{Bound_on_the_H1_level}. As in \cite{ChHaPavSei}, these estimates allow us to bound the contributions from the factors corresponding to distinguished and regular trees in Sub-subsection \ref{Estimating the factors corresponding to distinguished and regular trees} below. 

The proof of the derivation of the defocusing cubic nonlinear Schr\"{o}dinger equation on $\mathbb{T}^3$, i.e. of Theorem 2, is based on applying Theorem 1 to the solutions of the defocusing Gross-Pitaevskii hierarchy on $\mathbb{T}^3$, which were previously constructed in the work of Elgart, Erd\H{o}s, Schlein, and Yau \cite{EESY}. In particular, in this paper, the authors construct solutions without a statement about uniqueness. The spaces in which they work do not originally involve the trace norm. In Section \ref{derivation of defocusing cubic NLS on T3}, we show that the assumptions of Theorem 1 hold for these solutions. As an intermediate step, we will have to use the approximation argument stated in Subsection \ref{Approximation} in order to control higher powers of the $N$-body Hamiltonian as was needed in the assumptions in \cite{EESY}. We note that such a procedure was used in \cite{ESY2,ESY4,ESY5} in the non-periodic setting, as well as in \cite{KSS} in the periodic setting.

\subsection{Organization of the paper}
In Section \ref{Notation and known facts from Harmonic and Functional Analysis}, we will give the relevant notation and we will recall some important facts from Harmonic and Functional analysis.
In particular, in Subsection \ref{Function spaces}, we will recall the definition of the function spaces which we will use on $\mathbb{T}^3$. An important trilinear estimate from the work of Herr, Tataru, and Tzvetkov \cite{HTT} is recalled in Proposition \ref{Star}.
Subsection \ref{Density matrices} is devoted to notation and operations concerning density matrices. In Subsection \ref{Trace class}, we recall some important facts about trace class operators. Section \ref{Multilinear estimates} is devoted to the proof of the product estimate at the level of $H^1$ regularity, which is stated in Proposition \ref{Multilinear_estimate}.
In Section \ref{An unconditional uniqueness result}, we prove Theorem 1, which is stated in Theorem \ref{Unconditional_uniqueness}. The variant of the Quantum de Finetti Theorem which we will use is recalled in Subsection \ref{The Weak Quantum de Finetti Theorem}. The unconditional uniqueness result stated in Theorem 1 is proved in Subsection \ref{An unconditional uniqueness result}. Section \ref{derivation of defocusing cubic NLS on T3} is devoted to the proof of Theorem 2, which is stated as Theorem \ref{Cubic_NLS_T3_derivation}. In Subsection \ref{Approximation}, we recall the relevant approximation procedure which allows us to control  higher powers of $H_N$ applied to the initial data. A comparison of different forms of convergence which appear in the problem is given in Subsection \ref{A comparison of different forms of convergence}. In Subsection \ref{Proof of Theorem 2}, we give a proof of the derivation of the defocusing cubic NLS on $\mathbb{T}^3$, stated in Theorem 2.

\vspace{5mm}

\textbf{Acknowledgements:}
The author would like to thank Thomas Chen and Nata\v{s}a Pavlovi\'{c} for encouraging him to work on this problem. He would like to thank Benjamin Schlein for discussions concerning the references \cite{EESY,ESY1,ESY2,ESY3,ESY4,ESY5}, and Sebastian Herr for discussions concerning the references \cite{HTT,Strunk}, as well as for useful discussions on function spaces, which lead to the simplification of the exposition of Section \ref{Multilinear estimates}. The author is grateful to Antti Knowles for teaching a class on the dynamics of large quantum systems in the Fall of 2009. He would like to thank Gigliola Staffilani for many helpful comments. V.S. was supported by a Simons Postdoctoral Fellowship.

\section{Notation and known facts from Harmonic and Functional Analysis}
\label{Notation and known facts from Harmonic and Functional Analysis}

In our paper, given non-negative quantities $A$ and $B$, let us denote by $A \lesssim B$ an estimate of the form $A \leq CB$, for some constant $C>0$. If $C$ depends on a parameter $p$, we write the inequality as $A \lesssim_p B$ or we emphasize that $C=C(p)$. Throughout our paper, \emph{we will take the spatial domain $\Lambda$ to be the three-dimensional torus $\mathbb{T}^3=[0,2\pi]^3$}. 

\subsection{Function spaces}
\label{Function spaces}
Given $f \in L^2(\Lambda)$, we can define its Fourier transform as:
$$\widehat{f}(n):=\int_{\Lambda} f(x) e^{- i \langle x, n \rangle} dx.$$
Here, $n \in \mathbb{Z}^3$ and $\langle \cdot, \cdot \rangle$ denotes the inner product on $\mathbb{R}^3$.

Let us take the following convention for the Japanese bracket $\langle \cdot \rangle$ on $\Lambda$:
$$\langle x \rangle: =\sqrt{1+|x|^2}.$$
Given $s \in \mathbb{R}$, the Sobolev space $H^s(\Lambda)$ is the normed space corresponding to:
\begin{equation}
\notag
\|f\|_{H^s(\Lambda)}:=\big(\sum_{n \in \mathbb{Z}^3} |\widehat{f}(n)|^2 \cdot \langle n \rangle^{2s}\big)^{\frac{1}{2}}.
\end{equation}

Let us recall the periodic version of the Sobolev embedding estimate:
\begin{equation}
\label{Sobolev_embedding_torus}
\|\phi\|_{L^6(\Lambda)} \lesssim \|\phi\|_{H^1(\Lambda)}.
\end{equation}
The estimate \eqref{Sobolev_embedding_torus} can be deduced from the analogous estimate on $\mathbb{R}^3$ by using an extension argument. More precisely, given $\phi \in H^1(\Lambda)$, we define $\phi^{e}$ to be the periodic extension of $\phi$ to all of $\mathbb{R}^3$. We then let $F \in C_0^{\infty}(\mathbb{R}^3)$ be a function which is identically $1$ on $\Lambda$. Then, $\|F \cdot \phi^e\|_{L^6(\mathbb{R}^3)} \sim \|\phi\|_{L^6(\Lambda)}$ and $\|F \cdot \phi^{e}\|_{H^1(\mathbb{R}^3)} \lesssim \|\phi\|_{H^1(\Lambda)}$.
By the Sobolev embedding theorem on $\mathbb{R}^3$, we know that $\|F \cdot \phi^{e}\|_{L^6(\mathbb{R}^3)} \lesssim \|F \cdot \phi^{e}\|_{H^1(\mathbb{R}^3)}$. We can hence deduce \eqref{Sobolev_embedding_torus}.

Suppose that $N=2^j$ is a dyadic integer.
We will denote by $P_N$ the \emph{Littlewood-Paley projection} to frequencies which are of the order $N$. In order to do this, let us first consider $\Phi \in C_0^{\infty}(-2,2)$, a non-negative, even function, such that $\Phi \equiv 1$ on $[-1,1]$. We define the function $\Phi_N$ on $\mathbb{Z}^3$ by:

\begin{equation}
\notag
\begin{cases}
\Phi \big(\frac{|\xi|}{N}\big)-\Phi \big(\frac{2|\xi|}{N}\big),\,\mbox{for}\,\,N \geq 2.\\
\Phi(|\xi|),\,\,\mbox{for}\,\,N=1.
\end{cases}
\end{equation}

The operator $P_N$ is defined as:

\begin{equation}
\label{FrequencyProjection}
(P_N f)\,\,\widehat{}\,\,(\xi):=\Phi_N(\xi) \cdot \widehat{f}(\xi).
\end{equation}
In particular, 
\begin{equation}
\notag
supp \, (P_Nf)\,\,\widehat{}\, \subseteq \{|\xi| \sim N\}.
\end{equation}

\begin{proposition} (a consequence of Proposition 3.5 in \cite{HTT})

\label{Star}
There exists a universal constant $\delta>0$ such that for any dyadic integers $N_1,N_2,N_3$ with $N_1 \geq N_2 \geq N_3 \geq 1$, and for any finite interval $I \subseteq \mathbb{R}$, it is the case that:
\begin{equation}
\notag
\|P_{N_1}e^{it\Delta}f_1 \cdot P_{N_2}e^{it\Delta}f_2 \cdot P_{N_3}e^{it\Delta}f_3\|_{L^2(I \times \Lambda)}
\end{equation}
\begin{equation} 
\notag
\lesssim N_2 N_3 \max\{\frac{N_3}{N_1},\frac{1}{N_2}\}^{\delta} \cdot \|P_{N_1}f_1\|_{L^2(\Lambda)} \cdot \|P_{N_2}f_2\|_{L^2(\Lambda)} \cdot \|P_{N_3}f_3\|_{L^2(\Lambda)}.
\end{equation}
\end{proposition}
Let us note that the implied constant in the above estimate depends on the length $|I|$ of the interval $I$. It can be taken to be an increasing function of $|I|$. The result of Proposition \ref{Star} can be deduced by combining \cite[Proposition 3.5]{HTT} together with \cite[Proposition 2.10]{HTT} and the inclusion of spaces given in \cite[Proposition 2.8]{HTT}. We will omit the details. An analogous estimate was shown on the $3D$ irrational torus in \cite[Proposition 4.1]{Strunk}. 

\subsection{Density matrices}
\label{Density matrices}
Let us fix $k \in \mathbb{N}$.
A \emph{density matrix of order $k$} on $\Lambda$ is a function:
$$\gamma^{(k)}:\Lambda^k \times \Lambda^k \rightarrow \mathbb{C}.$$
It is also sometimes called a \emph{$k$-particle density matrix}.

Given $\gamma^{(k)} \in L^2(\Lambda^k \times \Lambda^k)$, we say that $\gamma^{(k)} \in L^2_{sym}(\Lambda^k \times \Lambda^k)$ if 
\begin{equation}
\notag
\gamma^{(k)}(x_{\sigma(1)},\ldots,x_{\sigma(k)};x'_{\sigma(1)},\ldots,x'_{\sigma(k)})=\gamma^{(k)}(x_1,\ldots,x_k;x'_1,\ldots,x'_k)
\end{equation} 
for all $(\vec{x}_k;\vec{x}'_k)=(x_1,\ldots,x_k;x'_1,\ldots,x'_k) \in \Lambda^k \times \Lambda^k$ and for all $\sigma \in S^k$. 
Similarly, if $\Psi \in L^2(\Lambda^k)$, we say that $\Psi \in L^2_{sym}(\Lambda^k)$ if:
\begin{equation}
\notag
\Psi(x_{\sigma(1)},\ldots,x_{\sigma(k)})=\Psi(x_1,\ldots,x_k)
\end{equation}
for all $\vec{x}_k=(x_1,\ldots,x_k) \in \Lambda^k$ and for all $\sigma \in S_k$. Here, we will always assume that it is clear from context whether the object we are considering is a function or a density matrix. In other words, we will be able to distinguish $L^2_{sym}(\Lambda^k \times \Lambda^k)$ from $L^2_{sym}(\Lambda^{2k})$.
Let us note that, if $\Psi \in L^2_{sym}(\Lambda^k)$, then $\big|\Psi \rangle \langle \Psi\big| \in L^2_{sym}(\Lambda^k \times \Lambda^k)$. We will sometimes consider $\gamma^{(k)}$ as an operator on $L^2(\Lambda^k)$ by means of identifying an integral operator with its kernel. We will use this convention throughout our work.

Given $\gamma^{(k)}$, a density matrix of order $k$ which belongs to $L^2(\Lambda^k \times \Lambda^k)$, its Fourier transform 
$(\gamma^{(k)})\,\,\widehat{}\,\,$ is defined as follows:
\begin{equation}
\label{gammak_Fourier_transform}
(\gamma^{(k)})\,\,\widehat{}\,\,(\vec{n}_k;\vec{n}'_k): = \int_{\Lambda^k \times \Lambda^k}\gamma^{(k)}(\vec{x}_k;\vec{x}'_k)\,\cdot 
e^{-i \cdot \sum_{j=1}^{k} \langle x_j, n_j \rangle + i \cdot \sum_{j=1}^{k} \langle x'_j, n'_j \rangle} \,d\vec{x}_k \, d\vec{x}'_k,
\end{equation}
for $(\vec{n}_k;\vec{n}'_k)=(n_1,\ldots,n_k;n'_1,\ldots,n'_k) \in \mathbb{Z}^{3k} \times \mathbb{Z}^{3k}$. 
The convention for \eqref{gammak_Fourier_transform} is consistent with our definition of factorized solutions \eqref{factorized_solution}.
We will also write $(\gamma^{(k)})\,\,\widehat{}\,\,$ as $\widehat{\gamma}^{(k)}$ for simplicity of notation.

The differential operator $i \partial_t + \big(\Delta_{\vec{x}_k}-\Delta_{\vec{x}'_k}\big)$ acts on density matrices of order $k$. Its associated \emph{free evolution operator} is denoted by $\mathcal{U}^{(k)}(t)$. More precisely, for $\gamma^{(k)}$, a density matrix of order $k$:
\begin{equation}
\notag
\mathcal{U}^{(k)}(t)\,\gamma^{(k)}:=e^{it \sum_{j=1}^{k} \Delta_{x_j}} \gamma^{(k)} e^{-it \sum_{j=1}^{k} \Delta_{x_j'}}.
\end{equation}
In this way, we obtain a solution to:
\begin{equation}
\notag
\Big(i \partial_t + (\Delta_{\vec{x}_k}-\Delta_{\vec{x}_k'})\Big)\,\mathcal{U}^{(k)}(t)\,\gamma^{(k)}=0.
\end{equation}

Given $\alpha \in \mathbb{R}$, we can use the Fourier transform and define the operation $S^{(k,\alpha)}$ of \emph{differentiation of order $\alpha$} on matrices of order $k$. Given $\gamma^{(k)}$ a density matrix of order $k$, we let $S^{(k,\alpha)}\gamma^{(k)}$ be the density matrix of order $k$ whose Fourier transform is given by:
\begin{equation}
\notag
\big(S^{(k,\alpha)} \gamma^{(k)} \big)\,\,\widehat{}\,\,(n_1,\ldots,n_k;n'_1,\ldots,n'_k):=
\end{equation}
\begin{equation}
\notag
\langle n_1 \rangle^{\alpha} \cdots \langle n_k \rangle^{\alpha} \cdot \langle n'_1 \rangle^{\alpha} \cdots \langle n'_k \rangle^{\alpha} \cdot \widehat{\gamma}^{(k)} (n_1, \ldots, n_k;n'_1, \ldots, n'_k).
\end{equation}

We now define the \emph{collision operator}.
Given $k \in \mathbb{N}$ and $j \in \{1,2,\ldots,k\}$, the collision operator $B_{j,k+1}$ acts linearly on density matrices of order $k+1$ as:

\begin{equation}
\label{Collision_Operator}
B_{j,k+1}\,(\gamma^{(k+1)}) = Tr_{\,k+1}\,\big[\delta(x_j-x_{k+1}),\gamma^{(k+1)}\big].
\end{equation}
\begin{equation}
\notag
=\int_{\Lambda} \Big(\delta(x_j-x_{k+1})\, \gamma^{(k+1)}(\vec{x}_k,x_{k+1};\vec{x}_k',x_{k+1})\, dx_{k+1}- \delta(x_j'-x_{k+1}) \,\gamma^{(k+1)}(\vec{x}_k,x_{k+1};\vec{x}_k',x_{k+1})\Big)\, dx_{k+1}.
\end{equation}
Here, $\delta$ denotes the Dirac delta function. We note that $B_{j,k+1}\,(\gamma^{(k+1)})$ is a density matrix of order $k$. We sometimes omit the parenthesis and write $B_{j,k+1}\,\gamma^{(k+1)}$. The \emph{full collision operator} $B_{k+1}$ is given by:

\begin{equation}
\label{Bk+1_sum}
B_{k+1}:=\sum_{j=1}^{k} B_{j,k+1}.
\end{equation}

\subsection{Some facts about trace class operators}
\label{Trace class}

Let $\mathcal{H}$ be a separable Hilbert space over $\mathbb{C}$ with inner product $\big( \cdot, \cdot \big)$.
Let $A$ be a bounded operator on $\mathcal{H}$ and let $\{\vec{e}_j,\,j \in \mathbb{N}\}$ be an orthonormal basis of $\mathcal{H}$. We say that A belongs to the \emph{trace class on $\mathcal{H}$} if the quantity:
\begin{equation}
\notag
\|A\|_{1}:=\sum_{j=1}^{+\infty} \big( (A^{*}A)^{\frac{1}{2}}\vec{e}_j, \vec{e}_j \big)
\end{equation}
is finite. If this is the case, the quantity $\|A\|_{1}$ can be shown to be independent of the choice of basis. The space of all trace-class operators on $\mathcal{H}$ is denoted by $\mathcal{L}^1(\mathcal{H})$.

Given $A \in \mathcal{H}$ and $\{\vec{e}_j,\,j \in \mathbb{N}\}$ as above, the \emph{trace of $A$:} 
\begin{equation}
\notag
Tr\,A:=\sum_{j=1}^{+\infty} \big( A \vec{e}_j,\vec{e}_j \big)
\end{equation}
is well-defined and is independent of the choice of basis.  This fact is proved in \cite[Theorem VI.24]{ReedSimon}. We will sometimes write $Tr A$ as $Tr_{\,\mathcal{H}} \,A$ in order to emphasize that the operator $A$ acts on the Hilbert space $\mathcal{H}$. We note that $\|A\|_1=Tr\,|A|=Tr\, \big(A^{*}A\big)^{\frac{1}{2}}$.
 
Let $\mathcal{K}(\mathcal{H})$ denote the space of all compact operators on $\mathcal{H}$. It is a normed vector space with respect to the operator norm $\|\cdot\|$. The following duality result then holds:

\begin{equation}
\label{duality}
\big(\mathcal{L}^1(\mathcal{H}),\|\cdot\|_1\big)=\big(\mathcal{K}(\mathcal{H}),\|\cdot\|\big)^{*}
\end{equation}
where the duality pairing is given by:
\begin{equation}
\label{duality_pairing}
(T,K) \in \mathcal{L}^1(\mathcal{H}) \times \mathcal{K}(\mathcal{H}) \mapsto Tr\,(TK).
\end{equation}
For the proof of the above result we refer the reader to \cite[Theorem VI.26]{ReedSimon}.

Given $N \in \mathbb{N}$, we let $\mathcal{H}^N:=\otimes_{j=1}^{N} \mathcal{H}$. If $A \in \mathcal{L}^1(\mathcal{H}^N)$, and $k \in \{1,2,\ldots,N\}$, we define $Tr_{\,k+1,\ldots,N}\,A$ to be the element of $\mathcal{L}^1(\mathcal{H}^k)$ obtained by taking the trace of $A$ in the last $N-k$ factors of $\mathcal{H}$. If $k>N$, we define $Tr_{\,k+1,\ldots,N}\,A$ to be equal to zero. We refer to this procedure as taking the partial trace.

Throughout the paper, we will primarily consider the case when $\mathcal{H}=L^2(\Lambda^k)$ with respect to Lebesgue measure for some $k \in \mathbb{N}$. In this case, we will write $\mathcal{L}^1(\mathcal{H})$ as $\mathcal{L}^1_k$ and $\mathcal{K}(\mathcal{H})$ as $\mathcal{K}_k$. These are subspaces of the space of density matrices of order $k$ on $\Lambda$. We observe that if $A \in \mathcal{L}^1_k$ is an integral operator with kernel $\gamma^{(k)} \in L^2(\Lambda^k \times \Lambda^k)$, then:

\begin{equation}
\notag
Tr\,A=\int_{\Lambda^k \times \Lambda^k} \gamma^{(k)}(\vec{x}_k;\vec{x}_k)\,d\vec{x}_k.
\end{equation}
For a more detailed discussion on trace class operators, we refer the reader to \cite[Sections 18 and 19]{Conway} and \cite[Section VI]{ReedSimon}.

\section{Multilinear estimates}
\label{Multilinear estimates}

We consider the spatial domain $\Lambda:=\mathbb{T}^3$. The main result of this section is:

\begin{proposition}
\label{Multilinear_estimate}
Suppose that $s \in [0,1]$ and suppose that $I$ is bounded time interval. Then, the following estimate holds:
\begin{equation}
\notag
\||\nabla_x|^s (e^{it\Delta} f_1 \cdot \overline{e^{it\Delta} f_2} \cdot e^{it\Delta} f_3)\|_{L^2(I \times \Lambda)} 
\end{equation}
\begin{equation}
\notag
\lesssim \min\{ \|f_1\|_{H^s} \cdot \|f_2\|_{H^1} \cdot \|f_3\|_{H^1},\|f_1\|_{H^1} \cdot \|f_2\|_{H^s} \cdot \|f_3\|_{H^1}, \|f_1\|_{H^1} \cdot \|f_2\|_{H^1} \cdot \|f_3\|_{H^s}\},
\end{equation}
whenever $f_1,f_2,f_3$ are functions on $\Lambda$ for which the right-hand side is well-defined. Here, the implied constant depends on $s$ and on the length of $I$.
\end{proposition}

\begin{proof}
Suppose that $f_1,f_2,f_3$ are as in the assumptions of the proposition. We will dyadically localize the factors according to \eqref{FrequencyProjection}.
Namely, for fixed dyadic integers $N_1,N_2,N_3$, we want to estimate the expression:

\begin{equation}
\notag
\||\nabla_x|^s (P_{N_1} e^{it\Delta} f_1 \cdot P_{N_2} \overline{e^{it\Delta} f_2} \cdot P_{N_3}e^{it\Delta} f_3)\|_{L^2(I \times \Lambda)} 
\end{equation}
which is:
\begin{equation}
\notag
\lesssim_s \max\{N_1,N_2,N_3\}^s \cdot \|P_{N_1} e^{it\Delta} f_1 \cdot P_{N_2} \overline{e^{it\Delta} f_2} \cdot P_{N_3}e^{it\Delta} f_3\|_{L^2(I \times \Lambda)} 
\end{equation}
\begin{equation}
\label{productNj}
=\max\{N_1,N_2,N_3\}^s \cdot \|P_{N_1}e^{it\Delta}f_1 \cdot P_{N_2}e^{it\Delta}f_2 \cdot P_{N_3} e^{it\Delta}f_3\|_{L^2(I \times \Lambda)}.
\end{equation}

Hence, we do not have to keep track of complex conjugates throughout the proof.
Let us define: 
\begin{equation}
\label{psi_j}
\psi_j:=e^{it\Delta}f_j\,\,\mbox{for $j=1,2,3$.}
\end{equation}   

Since $\|P_{N_1} \psi_1 \cdot P_{N_2} \psi_2 \cdot P_{N_3} \psi_3\|_{L^2(I \times \Lambda)}$ is symmetric under the permutation of the functions $f_j$, we may assume without loss of generality that $N_1 \geq N_2 \geq N_3 \geq 1$. In this case, the expression in \eqref{productNj} equals:

\begin{equation}
\notag
N_1^s \cdot \|P_{N_1} \psi_1 \cdot P_{N_2} \psi_2 \cdot P_{N_3} \psi_3\|_{L^2(I \times \Lambda)}=:K_{N_1,N_2,N_3}.
\end{equation}

We will now estimate:
$$K:=\mathop{\sum_{N_1,N_2,N_3}}_{N_1 \geq N_2 \geq N_3} K_{N_1,N_2,N_3}.$$
In order to do this, we will estimate each term $K_{N_1,N_2,N_3}$ by considering the different cases determined by the relative sizes of $N_1$ and $N_2$.

\vspace{5mm}

\textbf{Case 1: $N_1 \gg N_2$}

\vspace{5mm}

In this case, we will estimate $K_{N_1,N_2,N_3}$ by using duality. Namely, let use take $u \in L^2(I \times \Lambda)$ such that $\|u\|_{L^2(I \times \Lambda)}=1$ and let us note that:

$$\int_{I} \int_{\Lambda} u \cdot (P_{N_1} \psi_1 \cdot P_{N_2} \psi_2 \cdot P_{N_3} \psi_3) \,dx\,dt=
\int_{I} \int_{\Lambda} P_{N_0}u \cdot (P_{N_1} \psi_1 \cdot P_{N_2} \psi_2 \cdot P_{N_3} \psi_3) \,dx\,dt,$$
for some $N_0 \sim N_1$. More precisely, $N_0$ here denotes an appropriate dilation of the dyadic integer $N_1$ and we extend the definition of $P_{N_0}$ accordingly.

Consequently, we need to estimate the following sum:

$$K^1:=\mathop{\sum_{N_0,N_1,N_2,N_3}}_{N_0 \sim N_1, N_1 \gg N_2 \geq N_3} N_1^s \cdot \big|\int_{I} \int_{\Lambda} P_{N_0}u \cdot (P_{N_1} \psi_1 \cdot P_{N_2} \psi_2 \cdot P_{N_3} \psi_3)\,dx\,dt\big|$$
$$\leq \mathop{\sum_{N_0,N_1,N_2,N_3}}_{N_0 \sim N_1, N_1 \gg N_2 \geq N_3} N_1^s \cdot \|P_{N_0}u\|_{L^2(I \times \Lambda)} \cdot \|P_{N_1} \psi_1 \cdot P_{N_2} \psi_2 \cdot P_{N_3} \psi_3 \|_{L^2(I \times \Lambda)}$$
\begin{equation}
\label{Njs}
\lesssim_{|I|} \mathop{\sum_{N_0,N_1,N_2,N_3}}_{N_0 \sim N_1, N_1 \gg N_2 \geq N_3} \|P_{N_0}u\|_{L^2(I \times \Lambda)} \cdot N_1^s \cdot N_2 \cdot N_3 \cdot \max \big\{\frac{N_3}{N_1},\frac{1}{N_2}\big\}^{\delta} \cdot 
\end{equation}
$$\cdot \|P_{N_1}f_1\|_{L^2(\Lambda)} \cdot \|P_{N_2}f_2\|_{L^2(\Lambda)} \cdot \|P_{N_3}f_3\|_{L^2(\Lambda)}.$$
Here, we used the Cauchy-Schwarz inequality on $I \times \Lambda$, \eqref{psi_j} and Proposition \ref{Star}. By dyadic localization, the expression in  \eqref{Njs} is:
$$\lesssim \mathop{\sum_{N_0,N_1,N_2,N_3}}_{N_0 \sim N_1, N_1 \gg N_2 \geq N_3} \|P_{N_0}u\|_{L^2(I \times \Lambda)} \cdot \max \big\{\frac{N_3}{N_1},\frac{1}{N_2}\big\}^{\delta} \cdot \|P_{N_1}f_1\|_{H^s(\Lambda)} \cdot \|P_{N_2}f_2\|_{H^1(\Lambda)} \cdot \|P_{N_3}f_3\|_{H^1(\Lambda)}$$
$$\lesssim \mathop{\sum_{N_0,N_1,N_2,N_3}}_{N_0 \sim N_1, N_1 \gg N_2 \geq N_3} \|P_{N_0}u\|_{L^2(I \times \Lambda)} \cdot \Big\{\Big(\frac{N_3}{N_1}\Big)^{\frac{\delta}{2}} \cdot \Big(\frac{N_2}{N_1}\Big)^{\frac{\delta}{2}}+\frac{1}{N_2^{\frac{\delta}{2}} \cdot N_3^{\frac{\delta}{2}}}\Big\} \cdot$$
$$\cdot \|P_{N_1}f_1\|_{H^s(I)} \cdot \|P_{N_2}f_2\|_{H^1(\Lambda)} \cdot \|P_{N_3}f_3\|_{H^1(\Lambda)}.$$
The latter estimate follows from the fact that $N_1 \geq N_2 \geq N_3$.
We use the Cauchy-Schwarz inequality in $N_2$ and in $N_3$ to deduce that this quantity is:
$$\lesssim \mathop{\sum_{N_0,N_1}}_{N_0 \sim N_1}
\|P_{N_0}u\|_{L^2(I \times \Lambda)} \cdot \|P_{N_1}f_1\|_{H^s(\Lambda)} \cdot \|f_2\|_{H^1(\Lambda)} \cdot \|f_3\|_{H^1(\Lambda)}$$
which by using the Cauchy-Schwarz inequality in $N_0 \sim N_1$ is:
$$\lesssim \|u\|_{L^2(I \times \Lambda)} \cdot \|f_1\|_{H^s(\Lambda)} \cdot \|f_2\|_{H^1(\Lambda)} \cdot \|f_3\|_{H^1(\Lambda)}=\|f_1\|_{H^s(\Lambda)} \cdot \|f_2\|_{H^1(\Lambda)} \cdot \|f_3\|_{H^1(\Lambda)},$$
since $\|u\|_{L^2(I \times \Lambda)}=1$ by assumption.
In particular:
\begin{equation}
\label{K1A}
K^1 \lesssim \|f_1\|_{H^s(\Lambda)} \cdot \|f_2\|_{H^1(\Lambda)} \cdot \|f_3\|_{H^1(\Lambda)}.
\end{equation}


Since $s \in [0,1]$, we know that $N_1^s \cdot N_2 \cdot N_3 \leq N_1 \cdot N_2^s \cdot N_3$, and hence we can replace the factor of $N_1^s \cdot N_2 \cdot N_3$ in \eqref{Njs} by $N_1 \cdot N_2^s \cdot N_3$. By the above argument, it follows that:

\begin{equation}
\label{K1B}
K^1 \lesssim \|f_1\|_{H^1(\Lambda)} \cdot \|f_2\|_{H^s(\Lambda)} \cdot \|f_3\|_{H^1(\Lambda)}.
\end{equation}

Similarly, since $N_1^s \cdot N_2 \cdot N_3 \leq N_1 \cdot N_2 \cdot N_3^s$, it follows that:

\begin{equation}
\label{K1C}
K^1 \lesssim \|f_1\|_{H^1(\Lambda)} \cdot \|f_2\|_{H^1(\Lambda)} \cdot \|f_3\|_{H^s(\Lambda)}.
\end{equation}

We use duality and \eqref{K1A}, \eqref{K1B}, and \eqref{K1C} in order to deduce that:

\begin{equation}
\label{Case1bound}
\mathop{\sum_{N_1,N_2,N_3}}_{N_1 \gg N_2 \geq N_3} K_{N_1,N_2,N_3} \lesssim
\end{equation}
\begin{equation}
\notag
\min\{ \|f_1\|_{H^s} \cdot \|f_2\|_{H^1} \cdot \|f_3\|_{H^1},\|f_1\|_{H^1} \cdot \|f_2\|_{H^s} \cdot \|f_3\|_{H^1}, \|f_1\|_{H^1} \cdot \|f_2\|_{H^1} \cdot \|f_3\|_{H^s}\}.
\end{equation}
The implied constant depends on $s$ and the length of $I$ by construction. 

\vspace{5mm}

\textbf{Case 2: $N_1 \sim N_2$}

\vspace{5mm}

In this case, we will estimate:

$$K^{2}:=\mathop{\sum_{N_1,N_2,N_3}}_{N_1 \geq N_2 \geq N_3;\, N_1 \sim N_2} K_{N_1,N_2,N_3}=\mathop{\sum_{N_1,N_2,N_3}}_{N_1 \geq N_2 \geq N_3;\, N_1 \sim N_2} N_1^s \cdot \|P_{N_1}\psi_1 \cdot P_{N_2} \psi_2 \cdot P_{N_3} \psi_3\|_{L^2(I \times \Lambda)}.$$
By \eqref{psi_j} and Proposition \ref{Star}, this expression is:
\begin{equation}
\label{Njs2}
\lesssim_{|I|} \mathop{\sum_{N_1,N_2,N_3}}_{N_1 \geq N_2 \geq N_3;\, N_1 \sim N_2} N_1^s \cdot N_2 \cdot N_3 \cdot \max\{\frac{N_3}{N_1},\frac{1}{N_2}\}^{\delta} \cdot
\end{equation}
$$\cdot \|P_{N_1}f_1\|_{L^2(\Lambda)} \cdot \|P_{N_2}f_2\|_{L^2(\Lambda)} \cdot \|P_{N_3}f_3\|_{L^2(\Lambda)}.$$
Since $N_1 \sim N_2$ and $N_3 \geq 1$, this is:
$$\lesssim \mathop{\sum_{N_1,N_2,N_3}}_{N_1 \geq N_2 \geq N_3;\, N_1 \sim N_2} N_1^s \cdot N_2 \cdot N_3 \cdot \Big(\frac{N_3}{N_1}\Big)^{\delta} \cdot \|P_{N_1}f_1\|_{L^2(\Lambda)} \cdot \|P_{N_2}f_2\|_{L^2(\Lambda)} \cdot \|P_{N_3}f_3\|_{L^2(\Lambda)}$$
$$\lesssim \mathop{\sum_{N_1,N_2,N_3}}_{N_1 \geq N_2 \geq N_3;\, N_1 \sim N_2} \Big(\frac{N_3}{N_1}\Big)^{\delta} \cdot \|P_{N_1}f_1\|_{H^s(\Lambda)} \cdot \|P_{N_2}f_2\|_{H^1(\Lambda)} \cdot \|P_{N_3}f_3\|_{H^1(\Lambda)}.$$
By using the Cauchy-Schwarz inequality in $N_3$, we can bound this sum by:
$$\lesssim \mathop{\sum_{N_1,N_2}}_{N_1 \sim N_2} \|P_{N_1}f_1\|_{H^s(\Lambda)} \cdot \|P_{N_2}f_2\|_{H^1(\Lambda)} \cdot \|f_3\|_{H^1(\Lambda)}.$$
By using the Cauchy-Schwarz inequality in $N_1 \sim N_2$, this sum is:
\begin{equation}
\notag
\lesssim \|f_1\|_{H^s(\Lambda)} \cdot \|f_2\|_{H^1(\Lambda)} \cdot \|f_3\|_{H^1(\Lambda)}.
\end{equation}

As in Case 1, we can replace $N_1^s \cdot N_2 \cdot N_3$ by $N_1 \cdot N_2^s \cdot N_3$ and $N_1 \cdot N_2 \cdot N_3^s$ in \eqref{Njs2} and we can hence deduce that:

\begin{equation}
\label{Case2bound}
K^2=\mathop{\sum_{N_1,N_2,N_3}}_{N_1 \geq N_2 \geq N_3;\,N_1 \sim N_2} K_{N_1,N_2,N_3} \lesssim
\end{equation}
\begin{equation}
\notag
\min\{ \|f_1\|_{H^s} \cdot \|f_2\|_{H^1} \cdot \|f_3\|_{H^1},\|f_1\|_{H^1} \cdot \|f_2\|_{H^s} \cdot \|f_3\|_{H^1}, \|f_1\|_{H^1} \cdot \|f_2\|_{H^1} \cdot \|f_3\|_{H^s}\}.
\end{equation}
The implied constant depends on $s$ and the length of $I$ by construction.

The proposition now follows from the estimates \eqref{Case1bound} and \eqref{Case2bound}.

\end{proof}

\section{An unconditional uniqueness result}
\label{An unconditional uniqueness result}

In this section, we will use a version of the Quantum de Finetti Theorem in order to obtain an unconditional uniqueness result for the Gross-Pitaevskii hierarchy on $\mathbb{T}^3$. The unconditional uniqueness result will hold in a class of density matrices with regularity $\alpha = 1$, and hence for higher regularities. In particular, we will prove an analogue of the non-periodic unconditional uniqueness result proved by T. Chen, Hainzl, Pavlovi\'{c}, and Seiringer \cite{ChHaPavSei}.

\subsection{The Weak Quantum de Finetti Theorem}
\label{The Weak Quantum de Finetti Theorem}

Let $\mathcal{H}$ be a separable Hilbert space. We denote by $\mathcal{H}^k_{sym}:=\otimes_{sym}^k \,\mathcal{H}$ the associated bosonic $k$-particle space. More precisely, this is the space obtained from the quotient under the action of the symmetric group $S_k$ on
$\underbrace{\mathcal{H} \times \mathcal{H} \times \cdots \times \mathcal{H}}_\text{$k$ times}$ by $\sigma \cdot (h_1,h_2,\ldots,h_k):=(h_{\sigma(1)},h_{\sigma(2)},\ldots,h_{\sigma(k)}).$


We will use the following variant of the Quantum de Finetti Theorem, which is called the \emph{Weak Quantum de Finetti Theorem}:

\begin{theorem}{(Weak Quantum de Finetti \cite{AmmariNier1,AmmariNier2,LewinNamRougerie})}
\label{Weak Quantum de Finetti}

Suppose that the sequence $(\Gamma_N)_N$ satisfies the following assumptions for all $N \in \mathbb{N}$:

\begin{itemize}
\item[i)]  $\Gamma_N$ is a bounded self-adjoint operator on $\mathcal{H}^N_{sym}$.
\item[ii)] $\Gamma_N \geq 0$ as an operator on $\mathcal{H}^N_{sym}$.
\item[iii)] $Tr_{\mathcal{H}^N_{sym}} \, \Gamma_N =1$.
\end{itemize}
Moreover, suppose that for all $k \in \mathbb{N}$, the corresponding sequence of $k$-particle marginals $\gamma_N^{(k)}:=Tr_{k+1,\ldots,N}\, \Gamma_N$ converges to $\gamma^{(k)} \in \mathcal{H}^k_{sym}$ as $N \rightarrow \infty$ in the weak-$*$ topology of the trace class, i.e.
$\gamma_N^{(k)} \rightharpoonup^{*} \gamma^{(k)}$ in the trace class on $\mathcal{H}^k_{sym}$.

Under these assumptions, there exists a unique Borel probability measure supported on the unit ball $\mathcal{B}$ of $\mathcal{H}$, which is invariant under multiplication by complex numbers of modulus one such that, for all $k \in \mathbb{N}$:

\begin{equation}
\label{Quantum_de_Finetti_Star}
\gamma^{(k)}=\int_{\mathcal{B}} \big(|\phi \rangle \langle \phi|^{\otimes k} \big) \,d\mu(\phi).
\end{equation}
\end{theorem}
Theorem \ref{Weak Quantum de Finetti} is proved in \cite[Theorem 2.2]{LewinNamRougerie} and is based on equivalent results proved in \cite{AmmariNier1} and \cite{AmmariNier2}, where $\mu$ is called an \emph{Wigner measure}. The connection between Wigner measures and de Finetti measures is explained in \cite[Section 6.4]{Ammari}. Before stating our unconditional uniqueness result, let us make a few remarks concerning the general Weak Quantum de Finetti Theorem:

\begin{remark}
\label{Quantum_de_Finetti_Remark1}
By the invariance of $\mu$ under multiplication by complex numbers of modulus one, we mean that for all $z \in \mathbb{C}$ with $|z|=1$ and for all $\phi \in \mathcal{H}$, it is the case that $(z\cdot \mu)(\phi):=\mu(z \cdot \phi)=\mu(\phi)$.
\end{remark}

\begin{remark}
\label{Quantum_de_Finetti_Remark3}
There is also a Strong Quantum de Finetti Theorem. This is the original result obtained in the work of Hudson and Moody \cite{HudsonMoody}, and St\o rmer \cite{Stormer} in the context of $C^{*}$ algebras. In the context of density matrices results, analogues of this result were also obtained in the work of Ammari and Nier \cite{AmmariNier1,AmmariNier2}, and Lewin, Nam, and Rougerie \cite{LewinNamRougerie}. In the strong version, the density matrices $\gamma^{(k)} \in \mathcal{H}^{k}$ are not assumed to be weak-$*$ limits. Instead, they are assumed to have the property that $Tr_{\mathcal{H}^k}\,\gamma^{(k)}=1$ and that they satisfy the admissibility property $\gamma^{(k)}=Tr_{k+1}\,\gamma^{(k+1)}.$ In this case, it is possible to prove \eqref{Quantum_de_Finetti_Star} with the measure $\mu$ supported on the unit sphere $\mathcal{S}$ of $\mathcal{H}$. This strong version can be used to prove an unconditional uniqueness result with the corresponding assumptions, as was done in \cite{ChHaPavSei,HTX}. For our applications in Section \ref{derivation of defocusing cubic NLS on T3}, we will use the unconditional uniqueness result obtained from the Weak Quantum de Finetti Theorem in Theorem \ref{Unconditional_uniqueness} below.
\end{remark}

\subsection{An unconditional uniqueness result}
\label{An unconditional uniqueness result}

We henceforth consider the \emph{Gross-Pitaevskii hierarchy on $\Lambda=\mathbb{T}^3$:}

\begin{equation}
\label{GPhierarchy}
\begin{cases}
i \partial_t \gamma^{(k)} + \big(\Delta_{\vec{x}_k}-\Delta_{\vec{x}'_k}\big) \gamma^{(k)}=\lambda \cdot \sum_{j=1}^{k} B_{j,k+1}\big(\gamma^{(k+1)}\big) \\
\gamma^{(k)}\big|_{t=0}=\gamma_0^{(k)}. 
\end{cases}
\end{equation}
We will assume that $\lambda \in \{1,-1\}$.
The results of this section will apply to both the \emph{defocusing case} $\lambda=1$ and the \emph{focusing case} $\lambda=-1$. In Section \ref{derivation of defocusing cubic NLS on T3}, we will restrict our attention to the defocusing case.

Following \cite{ChHaPavSei}, we will use the spaces $\mathfrak{H}^{\alpha}$:

\begin{definition}
\label{mathfrakHalpha}
Given $\alpha \geq 0$, $\mathfrak{H}^{\alpha}$ denotes the set of all sequences $(\gamma^{(k)})_k$ such that, for each $k$:
\begin{itemize}
\item[i)]  $\gamma^{(k)} \in L^2_{sym}(\Lambda^k \times \Lambda^k)$ and $\gamma(\vec{x}_k,\vec{x}'_k)=\overline{\gamma^{(k)}(\vec{x}'_k;\vec{x}_k)}$ for all $(\vec{x}_k,\vec{x}'_k)$ in $\Lambda^k \times \Lambda^k$.
\item[ii)] $S^{(k,\alpha)}\gamma^{(k)} \in \mathcal{L}^1_k$.
\item[iii)] There exists $M>0$, independent of $k$, for which $Tr\big(|S^{(k,\alpha)}\gamma^{(k)}|\big) \leq M^{2k}$.
\end{itemize}
\end{definition}

Variants of such spaces were used previously in the work of Erd\H{o}s, Schlein, and Yau \cite{ESY1,ESY2,ESY3,ESY4,ESY5} and in related works. A different class of spaces, based on a Hilbert-Schmidt norm, was used in the work of Klainerman and Machedon \cite{KM}. 

We adopt the terminology from \cite{ChHaPavSei}:

\begin{definition}
\label{Mild_solution}
Given $T>0$, we say that $(\gamma^{(k)})_k=(\gamma^{(k)}(t))_k$ is a mild solution to the Gross-Pitaevskii hierarchy \eqref{GPhierarchy} in $L^{\infty}_{t \in [0,T]} \mathfrak{H}^{\alpha}$ if, for all $k \in \mathbb{N}$:
\begin{itemize}
\item[i)] $(\gamma^{(k)}(0))_k=(\gamma_0^{(k)})_k \in \mathfrak{H}^{\alpha}$.
\item[ii)] For all $t \in [0,T]$, the following integral equation is valid:
$$\gamma^{(k)}(t)=\mathcal{U}^{(k)}(t)-i\lambda \int_{0}^{t} \,\mathcal{U}^{(k)}(t-s)B_{k+1} \gamma^{(k+1)}(s)\,ds.$$
\item[iii)] $$\mathop{\sup}_{\,t \in [0,T]} Tr \big(|S^{(k,\alpha)} \gamma^{(k)}(t)|\big) \leq M^{2k}$$ for some $M>0$ which is independent of $k$ and $t$.
\end{itemize}
For point ii), we recall the definition of the full collision operator $B_{k+1}$ in \eqref{Bk+1_sum}.
\end{definition}

In what follows, let $S_t$ denote the flow map of the cubic nonlinear Schr\"{o}dinger equation:

\begin{equation}
\label{NLS_T3}
\begin{cases}
iu_t+\Delta u = \lambda \cdot |u|^2u\\
u \,\big|_{t=0}=\phi.
\end{cases}
\end{equation}
More precisely, we let $S_t(\phi)(x):=u(x,t)$.
We note that the Cauchy theory of this problem is well-understood \cite{B} and hence the map $S_t$ is well-defined.

\vspace{5mm}

With these definitions in mind, we will prove:
\begin{theorem}{(An unconditional uniqueness result)}
\label{Unconditional_uniqueness}

Let $(\gamma_0^{(k)})_k \in \mathfrak{H}^1$. Suppose that $(\gamma^{(k)}(t))_k$ is a mild solution to the Gross-Pitaevskii hierarchy \eqref{GPhierarchy} in $L^{\infty}_{t \in [0,T]} \mathfrak{H}^1$ with initial data $(\gamma^{(k)}(0))_k=(\gamma_0^{(k)})_k$ such that, for each $t \in [0,T]$, $\gamma^{(k)}(t)$ satisfies the assumptions of Theorem \ref{Weak Quantum de Finetti} with $\mathcal{H}=L^2(\Lambda)$. More precisely, we assume that there exist $\Gamma_{N,t} \in L^2_{sym}(\Lambda^N \times \Lambda^N)$ as in Theorem \ref{Weak Quantum de Finetti} such that $\gamma_{N,t}^{(k)}:=Tr_{k+1,\ldots,N}\, \Gamma_{N,t}$ satisfies $\gamma_{N,t}^{(k)} \rightharpoonup^{*} \gamma^{(k)}(t)$ as $N \rightarrow \infty$ in $\mathcal{L}^1_k$. Then, the conclusion is that $(\gamma^{(k)}(t))_k$ is uniquely determined by the initial data $(\gamma_0^{(k)})_k$.

Furthermore, suppose that $(\gamma_0^{(k)})_k$ satisfies, for all $k \in \mathbb{N}$:
\begin{equation}
\label{initial_data}
\gamma_0^{(k)}=\int_{\mathcal{B}}  \big(|\phi \rangle \langle \phi|^{\otimes k}\big)  \,d\mu(\phi)
\end{equation}
for $\mu$ a Borel probability measure which is invariant under multiplication by complex numbers of modulus one and which is supported on the unit ball $\mathcal{B}$ of $L^2(\Lambda)$. Then, for all $k \in \mathbb{N}$ and for all $t \in [0,T]$, it is the case that:
\begin{equation}
\label{gammak_formula}
\gamma^{(k)}(t)=\int_{\mathcal{B}} \big(|S_t(\phi) \rangle \langle S_t(\phi)|^{\otimes k}\big) \,d\mu(\phi).
\end{equation}
\end{theorem}

\begin{remark}
\notag
Let us note that the a unique measure $\mu$ satisfying \eqref{initial_data} exists by Theorem \ref{Weak Quantum de Finetti}.
\end{remark}

\begin{remark}
We note that, by subdividing the time interval, we can assume without loss of generality that $T \in [0,1]$. 
\end{remark}
The remainder of this section is devoted to the proof of Theorem \ref{Unconditional_uniqueness}. Our method will follow the combinatorial graph expansion from \cite{ChHaPavSei}. The main difference from the non-periodic setting is that we will estimate the terms coming from the \emph{distinguished vertices} by using Proposition \ref{Multilinear_estimate}. More precisely, we will have to give a different proof of the $L^2$ estimate \eqref{Bound_on_the_L2_level} and of the $H^1$ estimate \eqref{Bound_on_the_H1_level} below. Namely, the proof of these estimates in the non-periodic setting \cite[Lemma 7.1]{ChHaPavSei} relies on the use of Strichartz estimates on $\mathbb{R}^3$, which are not available on $\mathbb{T}^3$. We overcome this difficulty by applying Proposition \ref{Multilinear_estimate}.
In addition, in bounding the trilinear term $\||\phi|^2 \phi\|_{L^2_x}$, we need to recall the Sobolev embedding on $\mathbb{T}^3$ in \eqref{Sobolev_embedding_torus}.
For completeness of the exposition, we will recall the main ideas of the graph expansion and how it is used to prove unconditional uniqueness results. For the full details of the construction and the methods, we refer the interested reader to \cite{ChHaPavSei}.

Let us suppose that $(\gamma^{(k)}_1(t))_k$ and $(\gamma^{(k)}_2(t))_k$ are solutions to \eqref{GPhierarchy} with initial data $(\gamma_0^{(k)})_k$ which satisfy the assumptions of the theorem. Let $\gamma^{(k)}:=\gamma_1^{(k)}-\gamma_2^{(k)}$. Our goal is to show that $\gamma^{(k)}$ is identically zero on $[0,T]$ for all $k \in \mathbb{N}$. In particular, it suffices to show that:

\begin{equation}
\label{claim}
Tr \big(|\gamma^{(k)}(t)|\big)=0
\end{equation}
for all $k \in \mathbb{N}$ and for all $t \in [0,T]$.



By Theorem \ref{Weak Quantum de Finetti}, there exist unique time-dependent Borel probability measures $\mu^{(1)}_t$ and $\mu^{(2)}_t$, which are invariant under multiplication by complex numbers of modulus one and which are supported on the unit ball $\mathcal{B}$ of $L^2(\Lambda)$ such that for all $k \in \mathbb{N}$ and $t \in [0,T]:$

\begin{equation}
\notag
\gamma_1^{(k)}(t)=\int_{\mathcal{B}} \big(|\phi \rangle \langle \phi|^{\otimes k}\big) \,d\mu^{(1)}_t (\phi) 
\end{equation}
and
\begin{equation}
\notag
\gamma_2^{(k)}(t)=\int_{\mathcal{B}} \big(|\phi \rangle \langle \phi|^{\otimes k}\big) \,d\mu^{(2)}_t (\phi).
\end{equation}
Hence, we can write:
\begin{equation}
\label{3c}
\gamma^{(k)}(t)=\int_{\mathcal{B}} \big(|\phi \rangle \langle \phi|^{\otimes k}\big) \, d\mu_t (\phi),
\end{equation}
where
\begin{equation}
\notag
\mu_t:=\mu^{1}_t-\mu^{2}_t.
\end{equation}
is a signed Borel measure on the unit ball of $L^2(\Lambda)$.
The assumption that  $(\gamma^{(k)}_1(t))_k$ and $(\gamma^{(k)}_2(t))_k \in L^{\infty}_{t \in [0,T]} \mathfrak{H}^1$ implies that for $j=1,2$, there exists $M_j>0$, such that for all $k \in \mathbb{N}$ and for all $t \in [0,T]$: 

\begin{equation}
\label{Growth_Bound}
Tr(|\gamma^{(k)}_j(t)|)=\int_{\mathcal{B}} \|\phi\|_{H^1}^{2k} \, d\mu^{(j)}_t (\phi) \leq M_j^k.
\end{equation}


Let us note that, by Duhamel expansion:

\begin{equation}
\label{Duhamel_Expansion_r}
\gamma^{(k)}(t)= \int_{0}^{t} \int_{0}^{t_1} \cdots \int_{0}^{t_{r-1}} (-i\lambda)^r \cdot \, J^k(\underline{t}_r) \, d \underline{t}_r,
\end{equation}
where: 
$$\underline{t}_r:=(t_1,t_2,\ldots,t_r)$$ 
and the integrand is given by:
\begin{equation}
\label{Duhamel_integrand}
J^k(\underline{t}_r):=\mathcal{U}^{(k)}\,(t-t_1)\,B_{k+1}\,\,\mathcal{U}^{(k+1)}\,(t_1-t_2)\,B_{k+2}\cdots
\end{equation}
$$\cdots \,\mathcal{U}^{(k+r-1)}(t_{r-1}-t_r) \, B_{k+r} \gamma^{(k+r)}(t_r).
$$

We can write:
$$J^k(\underline{t}_r)=\sum_{\rho \in \mathcal{M}_{k,r}} J^k(\rho;\underline{t}_r),$$
where:
$$J^k(\rho;\underline{t}_r):=\mathcal{U}^{(k)}(t-t_1)\, B_{\rho(k+1),\,k+1} \, \mathcal{U}^{(k+1)}(t_1-t_2) \,$$
$$B_{\rho(k+2),\,k+2} \cdots \,\mathcal{U}^{(k+r-1)} (t_{r-1}-t_r) \,B_{\rho(k+r),\,k+r} \gamma^{(k+r)}(t_r),$$
for a fixed $\rho:\{k+1,k+2,\ldots,k+r\} \rightarrow \{1,2,\ldots,k+r-1\}$ such that $\rho(j)<j$ for all $j \in \{k+1,k_2,\ldots,k+r\}$. $\mathcal{M}_{k,r}$ denotes the set of all such mappings $\rho$. 

It was noted in the work of Klainerman and Machedon \cite[Section 3]{KM} that, by using a combinatorial \emph{boardgame argument}, one can define an equivalence relation $``\simeq"$ on $\mathcal{M}_{k,r}$ under which every $\rho \in \mathcal{M}_{k,r}$ is equivalent to a unique \emph{upper echelon} element $\sigma$, i.e. to $\sigma \in \mathcal{M}_{k,r}$ which is monotonically increasing. 
Let $\mathcal{N}_{k,r}$ denote the set of all upper echelon elements in $\mathcal{M}_{k,r}$. From the analysis in \cite[Section 3]{KM}, one can deduce that: 
$$\# \mathcal{N}_{k,r} \leq 2^{k+r}$$ 
and:
\begin{equation}
\label{gammak_boardgame_argument}
\gamma^{(k)}(t)=\sum_{\rho \in \mathcal{M}_{k,r}} \int_{0}^{t} \int_{0}^{t_1} \cdots \int_{0}^{t_{r-1}} (-i\lambda)^r  \cdot J^k(\rho;\underline{t}_r) \, d \underline{t}_r=
\end{equation}
$$=\sum_{\sigma \in \mathcal{N}_{k,r}} \int_{D(\sigma,t)} (-i\lambda)^r \cdot J^{k} (\sigma;\underline{t}_r) \,d\underline{t}_r$$
for $D(\sigma,t)$ a union of simplices contained in $[0,t]^r$, whose interiors are mutually disjoint. Similar arguments were given in a different context in \cite{ESY1,ESY2,ESY4,ESY5}.

We now use \eqref{3c}, \eqref{Duhamel_integrand}, and \eqref{gammak_boardgame_argument} to write:

\begin{equation}
\label{gamma_k_sum}
\gamma^{(k)}(t)=(-i\lambda)^r  \cdot \sum_{\sigma \in \mathcal{N}_{k,r}} \int_{D(\sigma,t)} \int_{\mathcal{B}}  J^k(\phi;\sigma;t,t_1,\ldots,t_r) \,d\mu_t(\phi)\,d\underline{t}_r
\end{equation}
where, for $\sigma \in \mathcal{N}_{k,r}$ and $\phi \in \mathcal{B}$, we write:

\begin{equation}
\label{Jk}
J^k(\phi;\sigma;t,t_1,\ldots,t_r;\vec{x}_k;\vec{x}'_k):=
\end{equation}
$$\mathcal{U}^{(k)}(t-t_1)\,B_{\sigma(k+1),\,k+1}\,\mathcal{U}^{(k+1)}(t_1-t_2)\cdots \,\mathcal{U}^{(k+r-1)}(t_{r-1}-t_r) \,B_{\sigma(k+r),\,k+r} (|\phi \rangle \langle \phi|^{\otimes(k+r)})(\vec{x}_k;\vec{x}'_k).$$
In \eqref{Jk}, we use the convention as in \cite{ChHaPavSei} that the time variable $t_{\ell}$ is attached to the interaction operator $B_{\sigma(k+\ell),\,k+\ell}$.

\subsubsection{\textbf{A reformulation of the Duhamel expansion following \cite{ChHaPavSei}}}
\medskip
In \cite[equation (4.23)]{ChHaPavSei}, it is noted that $J^k(\phi;\sigma;t,t_1,\ldots,t_r;\vec{x}_k;\vec{x}'_k)$ has the following product structure:
\begin{equation}
\label{product_structure}
J^k(\phi;\sigma;t,t_1,\ldots,t_r;\vec{x}_k;\vec{x}'_k)=\prod_{j=1}^{k} J^1_j (\phi;\sigma_j;t,t_{\ell_{j,1}}, \ldots, t_{\ell_{j,m_j}};x_j;x'_j).
\end{equation}
Here, $\prod$ denotes a tensor product which is invariant under permutation of the factors and each $J^1_j$ is an expression of the same form as in \eqref{Jk}. The interaction operators which appear in $J^1_j$ are labelled ``internally'' with the maps $\sigma_j, j=1, \ldots, k$. Relative to $J^1_j$, $\sigma_j$ is in upper echelon form.
We note that the construction of \cite{ChHaPavSei} is originally given on $\mathbb{R}^3$, but carries over to the periodic setting without any changes.

By using \eqref{gamma_k_sum}, \eqref{product_structure}, and the triangle inequality, it follows that:

\begin{equation}
\label{gamma_k_estimate}
Tr \big(|\gamma^{(k)}(t)|\big) \leq C^{k+r} \cdot \sum_{\ell=1}^{2} \sup_{\sigma \in \mathcal{N}_{k,r}} \int_{[0,t]^r} \int_{\mathcal{B}} \,
\end{equation}
$$\prod_{j=1}^{k} Tr \big(|J^1_j(\phi;\sigma_j;t,t_{\ell_{j,1}},\cdots,t_{\ell{j,m_j}})|\big)\,d\mu^{(\ell)}_{t_r} (\phi) \,d \underline{t}_r.$$

We note that the factor of $C^{k+r}$ for some universal constant $C>0$ comes from the boardgame argument.
\subsubsection{\textbf{An explicit example of the graph expansion}}
\label{Example}

Before proceeding further, let us give an explicit example of the product formula given by \eqref{product_structure}.

\begin{example}
Let us consider the case when $k=3$ and $r=5$.  Suppose that $\sigma(4)=1, \sigma(5)=2, \sigma(6)=3, \sigma(7)=4$, and $\sigma(8)=6$. Then $\sigma \in \mathcal{N}_{3,5}$ and it corresponds to the upper-echelon matrix in the terminology from \cite{KM}:

\begin{equation}
\notag
\begin{bmatrix}
\mathbf{B_{1,4}} & B_{1,5} & B_{1,6} & B_{1,7} & B_{1,8} \\
B_{2,4}  &  \mathbf{B_{2,5}} & B_{2,6} & B_{2,7} & B_{2,8}\\
B_{3,4} & B_{3,5} & \mathbf{B_{3,6}} & B_{3,7} & B_{3,8} \\
0 & B_{4,5} & B_{4,6} & \mathbf{B_{4,7}} & B_{4,8} \\
0 & 0 & B_{5,6} & B_{5,7} & B_{5,8} \\
0 & 0 & 0 & B_{6,7} & \mathbf{B_{6,8}} \\
0 & 0 & 0 & 0 & B_{7,8} \\
\end{bmatrix}
\end{equation}
\end{example}

In particular, we consider:
\begin{equation}
\notag
J^{3}(\phi;\sigma;t,t_1,t_2,t_3,t_4,t_5)= \,\mathcal{U}^{(3)}_{\,0,1} \, B_{1,4} \,\,\mathcal{U}^{(4)}_{\,1,2} \,\, B_{2,5} \, \mathcal{U}^{(5)}_{\,2,3} \, B_{3,6} \,\, \mathcal{U}^{(6)}_{\,3,4} \, B_{4,7} \,\, \mathcal{U}^{(7)}_{\,4,5} \, B_{6,8} \big(|\phi \rangle \langle \phi)^{\otimes k}\big),
\end{equation}
which is henceforth written as $J^3$.
Here, as in \cite{ChHaPavSei}, we abbreviate $\mathcal{U}^{(\ell)}_{\,j,j'}:=\,\mathcal{U}^{(\ell)}(t_j-t_j')$, where $t_0:=t$. Moreover, we write $$|\phi \rangle \langle \phi|^{\otimes k}= \bigotimes_{\ell=1}^{8} u_{\ell}$$ 
in order to distinguish the factors.

Let us first compute:

\begin{equation}
\label{E1}
B_{6,8} (\otimes_{\ell=1}^{8} u_{\ell})=(\otimes_{\ell=1}^{5} u_{\ell}) \otimes \theta_5 \otimes u_7,
\end{equation}
where $\theta_5$ corresponds to the fifth collision from the left in the formula for $J^3$. In other words, we take:

\begin{equation}
\label{E2}
\theta_5:=B_{1,2}(u_6 \otimes u_8).
\end{equation}

Then $B_{4,7} \,\, \mathcal{U}^{(7)}_{\,4,5}$ applied to \eqref{E1} equals:

$$B_{4,7} \,\, \mathcal{U}^{(7)}_{\,4,5} \big( (\otimes_{\ell=1}^{5} u_{\ell}) \otimes \theta_5 \otimes u_7 \big)=$$
\begin{equation}
\label{E3}
=\mathcal{U}^{(3)}_{\,4,5} (\otimes_{\ell=1}^{3} u_{\ell}) \otimes \theta_4 \otimes \mathcal{U}^{(1)}_{\,4,5} \, u_5 \otimes \mathcal{U}^{(1)}_{\,4,5}\, \theta_5,
\end{equation}
where:
\begin{equation}
\label{E4}
\theta_4:=B_{1,2} \,\,\mathcal{U}^{(2)}_{\,4,5}\, (u_4 \otimes u_7).
\end{equation}
$B_{3,6} \,\, \mathcal{U}^{(6)}_{\,3,4}$ applied to \eqref{E3} equals:
$$B_{3,6}\,\,\mathcal{U}^{(6)}_{\,3,4} \big(\,\, \mathcal{U}^{(3)}_{\,4,5} (\otimes_{\ell=1}^{3} u_{\ell}) \otimes \theta_4 \otimes \mathcal{U}^{(1)}_{\,4,5} \, u_5 \otimes \mathcal{U}^{(1)}_{\,4,5}\, \theta_5$$
\begin{equation}
\label{E5}
=\mathcal{U}^{(1)}_{\,3,5} \, u_1 \otimes \mathcal{U}^{(1)}_{\,3,5} \,u_2 \otimes \theta_3 \otimes \mathcal{U}^{(1)}_{\,3,4} \, \theta_4 \otimes \mathcal{U}^{(1)}_{\,3,5} \, u_5\big).
\end{equation}
Here:
\begin{equation}
\label{E6}
\theta_3:=B_{1,2} \big(\,\mathcal{U}^{(1)}_{\,3,5} \,u_3 \otimes \mathcal{U}^{(1)}_{\,3,5}\,\theta_5 \big).
\end{equation}
In particular, $B_{2,5} \, \mathcal{U}^{(5)}_{\,2,3}$ applied to \eqref{E5} equals:
$$B_{2,5} \,\, \mathcal{U}^{(5)}_{\,2,3} \big(\,\, \mathcal{U}^{(1)}_{\,3,5} \, u_1 \otimes \mathcal{U}^{(1)}_{\,3,5} \,u_2 \otimes \theta_3 \otimes \mathcal{U}^{(1)}_{\,3,4} \, \theta_4 \otimes \mathcal{U}^{(1)}_{\,3,5} \, u_5 \big)=$$
\begin{equation}
\label{E7}
=\mathcal{U}^{(1)}_{\,2,5} \,u_1 \otimes \theta_2 \otimes \mathcal{U}^{(1)}_{\,2,3} \, \theta_3 \otimes \mathcal{U}^{(1)}_{\,2,4} \, \theta_4,
\end{equation}
where:
\begin{equation}
\label{E8}
\theta_2:=B_{1,2} \big(\, \mathcal{U}^{(1)}_{\,2,5} \,u_2 \otimes \mathcal{U}^{(1)}_{\,2,5}\, u_5\big).
\end{equation}
$B_{1,4}\,\,\mathcal{U}^{(4)}_{\,1,2}$ applied to \eqref{E7} then equals:
$$B_{1,4}\,\,\mathcal{U}^{(4)}_{\,1,2} \,\, \big(\,\,\mathcal{U}^{(1)}_{\,2,5} \,u_1 \otimes \theta_2 \otimes \mathcal{U}^{(1)}_{\,2,3} \, \theta_3 \otimes \mathcal{U}^{(1)}_{\,2,4} \, \theta_4\big)=$$
\begin{equation}
\label{E9}
=\theta_1 \otimes \mathcal{U}^{(1)}_{\,1,2} \,\theta_2 \otimes \mathcal{U}^{(1)}_{\,1,3} \,\theta_3,
\end{equation}
where:
\begin{equation}
\label{E10}
\theta_1:=B_{1,2} \big(\,\mathcal{U}^{(1)}_{\,1,5} \,u_1 \otimes \mathcal{U}^{(1)}_{\,1,4}\,\theta_4 \big).
\end{equation}
Finally, we note that $J^3$ equals $\mathcal{U}^{(3)}_{\,0,1}$ applied to \eqref{E9}.

In particular:

$$J^3= \mathcal{U}^{(3)}_{\,0,1} \big(\theta_1 \otimes \mathcal{U}^{(1)}_{\,1,2} \,\theta_2 \otimes \mathcal{U}^{(1)}_{\,1,3} \,\theta_3\big)=$$
\begin{equation}
\notag
=\mathcal{U}^{(1)}_{\,0,1} \, \theta_1 \otimes \mathcal{U}^{(1)}_{\,0,2} \,\theta_2 \otimes \mathcal{U}^{(1)}_{\,0,3} \,\theta_3 =: J^1_1 \otimes J^1_2 \otimes J^1_3.
\end{equation}

Let us explicitly compute the factors above. By using \eqref{E4} and \eqref{E10}, we observe:

\begin{equation}
\notag
J^1_1=\mathcal{U}^{(1)}_{\,0,1} \, B_{1,2} \,\,\mathcal{U}^{(2)}_{\,1,4} \, B_{2,3} \,\, \mathcal{U}^{(3)}_{\,4,5} \, (u_1 \otimes u_4 \otimes u_7).
\end{equation}
From \eqref{E8}, we deduce:
\begin{equation}
\notag
J^1_2=\mathcal{U}^{(1)}_{\,0,2} \, \theta_2 = \mathcal{U}^{(1)}_{\,0,2} \, B_{1,2} \,\,\mathcal{U}^{(2)}_{\,2,5}\,(u_2 \otimes u_5).
\end{equation}
Finally, by using \eqref{E2} and \eqref{E6}, it follows that:
\begin{equation}
\notag
J^1_3=\mathcal{U}^{(1)}_{\,0,3} \, B_{1,2} \, \mathcal{U}^{(2)}_{\,3,5} \, B_{2,3} (u_3 \otimes u_6 \otimes u_8).
\end{equation}

Following the terminology of \cite{ChHaPavSei}, we note that $J^1_3$ is the \emph{distinguished factor} in the product, since the collision operator $B_{2,3}$ is not followed by a free evolution. Furthermore as in \cite[Section 5]{ChHaPavSei}, one can associate to the product \eqref{product_structure} a graph. In our example, the graph would be:

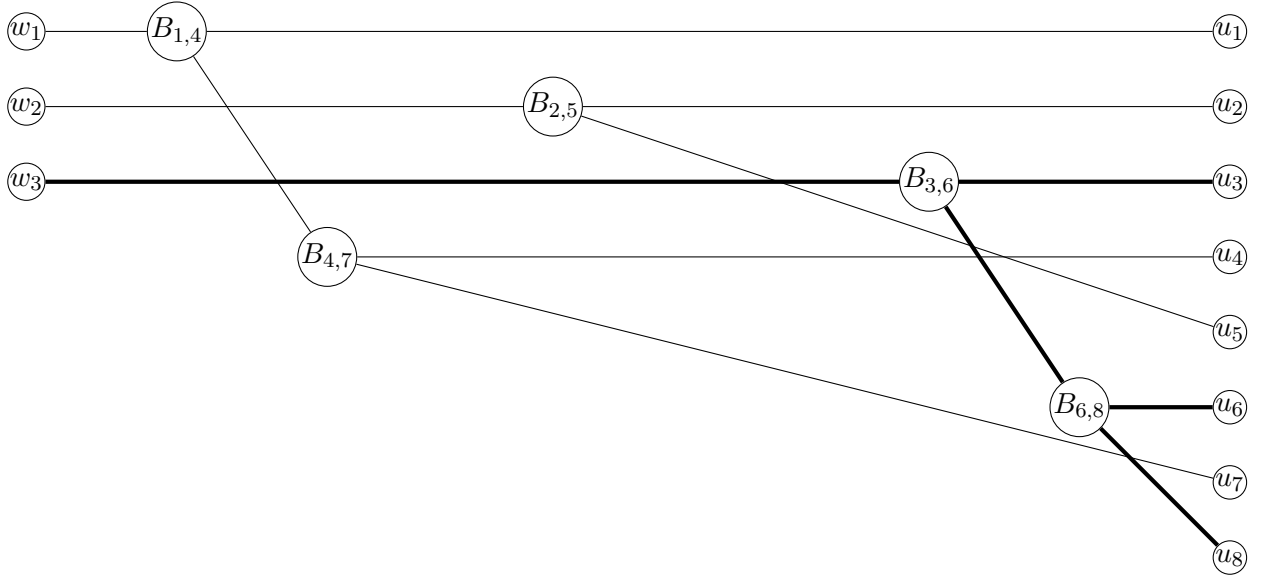
\vspace{5mm}

\begin{tikzpicture}[scale=0.2]
\draw node (w1) at (-5,5) {$w_1$};
\draw node (w2) at (-5,0) {$w_2$};
\draw node (w3) at (-5,-5) {$w_3$};
\draw node (B14) at (5,5) {$B_{1,4}$};
\draw node (B25) at (30,0) {$B_{2,5}$};
\draw node (B47) at (15,-10) {$B_{4,7}$};
\draw node (B36) at (55,-5) {$B_{3,6}$};
\draw node (B68) at (65,-20) {$B_{6,8}$};
\draw node (u1) at (75,5) {$u_1$};
\draw node (u2) at (75,0) {$u_2$};
\draw node (u3) at (75,-5) {$u_3$};
\draw node (u4) at (75,-10) {$u_4$};
\draw node (u5) at (75,-15) {$u_5$};
\draw node (u6) at (75,-20) {$u_6$};
\draw node (u7) at (75,-25) {$u_7$};
\draw node (u8) at (75,-30) {$u_8$};
\draw (B25)--(u5);
\draw (w1)--(B14);
\draw (B14)--(u1);
\draw (w2)--(B25);
\draw (B25)--(u2);
\draw [ultra thick] (w3)--(B36);
\draw [ultra thick] (B36)--(u3);
\draw (B14)--(B47);
\draw (B47)--(u4);
\draw [ultra thick] (B36)--(B68);
\draw [ultra thick] (B68)--(u6);
\draw (B47)--(u7);
\draw [ultra thick] (B68)--(u8);
\end{tikzpicture}

\vspace{5mm}

In the above graph, the number of subtrees equals the number of distinct factors $k$ in the product \eqref{product_structure}, which in this case is $3$. The roots of these trees are labeled as $w_1,w_2,w_3$. There are $k=5$ internal vertices, which correspond to the collision operators. The number of leaves in the graph equals $k+r=8$, which is the total number of factors of $|\phi \rangle \langle \phi|$ that occur in \eqref{product_structure}. These leaves are labeled as $u_1,u_2,\ldots,u_8$. Each internal node is obtained from a collision operator in the expression for $J^3$ in a way which will be explained more precisely below. The subtree corresponding to the distinguished factor $J^1_3$ is highlighted.

\subsubsection{\textbf{The precise definition of the graph expansion (after \cite{ChHaPavSei})}}
Let us now recall the precise definition of the graph structure following \cite[Section 5]{ChHaPavSei}. One assigns to the expression \eqref{Jk} and the product structure \eqref{product_structure} a union of $k$ disjoint binary tree graphs. As is noted in \cite{ChHaPavSei}, these graphs have already appeared as substructures associated to the graphs in the more involved analysis of \cite{ESY5}. More precisely, one assigns:

\vspace{5mm}

$1)$ An \emph{internal vertex} $v_{\ell}$ associated to each collision operator $B_{\sigma(k+\ell),k+\ell}$, for $\ell=1,2,\ldots,r$. Hence, the time variable $t_{\ell}$ can be thought of as being attached to the vertex $v_{\ell}$.

$2)$ A \emph{root vertex} $w_j$ to each factor $J^1_j$, for $j=1,2,\ldots,k$.

$3)$ A \emph{leaf vertex} $u_i$ to the factor $\big(|\phi \rangle \langle \phi|\big)(x_i;x'_i)$ occurring in $|\phi \rangle \langle \phi|^{\otimes (k+r)}(\vec{x}_{k+r},\vec{x}'_{k+r})$. Here $i=1,2,\ldots,k+r$.

\vspace{5mm}

Having defined the vertices of the graph, one defines the equivalence relation $``\sim"$ of connectivity between vertices. In other words, if two vertices are equivalent under $\sim$, we will draw an edge between them. The equivalence relation $\sim$ on the above defined set of vertices is given as follows:

\vspace{5mm}

$1)$ Suppose that $j \in \{1,2,\ldots,k\}$ is given. If $\ell$  is the \emph{smallest element} of $\{1,2,\ldots,r\}$ such that $\sigma(k+\ell)=j$, then the root vertex $w_j$ is connected to the internal vertex $v_{\ell}$, i.e. these two vertices are equivalent. We say that $w_j$ is the \emph{parent vertex} of $v_{\ell}$ and that $v_{\ell}$ is the \emph{child vertex} of $w_j$.

$2)$ If there exists no internal vertex connected to $w_j$ as above, in other words if $\sigma(k+\ell) \neq j$ for all $\ell \in \{1,2,\ldots,r\}$, then the root vertex $w_j$ is connected to the leaf vertex $u_j$. We say that $w_j$ is the \emph{parent vertex} of $u_j$ and $u_j$ is the \emph{child vertex} of $w_j$.

$3)$ Suppose that $\ell \in \{1,2,\ldots,r\}$. If there exists $\ell' \in \{1,2,\ldots,r\}$ with $\ell'>\ell$ such that $k+\ell=\sigma(k+\ell')$ or $\sigma(k+\ell)=\sigma(k+\ell')$, then we say that $v_{\ell} \sim v_{\ell'}.$
In this case, it is said that $v_{\ell}$ is a \emph{parent vertex} of $v_{\ell'}$ and that $v_{\ell'}$ is a \emph{child vertex} of $v_{\ell}$. The two child vertices of $v_{\ell}$ will be denoted by $v_{\kappa_{-}(\ell)}$ and $v_{\kappa_{+}(\ell)}$, where by convention one takes $\kappa_{-}(\ell) < \kappa_{+}(\ell)$.

$4)$ If there exists no internal vertex $v_{\ell'}$ with $\ell'>\ell$ and $k+\ell=\sigma(k+\ell')$ as above, then $v_{\ell}$ is connected to the leaf vertex $u_{\ell}$. In this case, we say that $v_{\ell}$ is the \emph{parent vertex} of $u_{\ell}$ and that $u_{\ell}$ is a \emph{child vertex} of $v_{\ell}$. Furthermore, if there exists no internal vertex $v_{\ell'}$ such that $\sigma(k+\ell)=\sigma(k+\ell')$, then $v_{\ell}$ is connected to the leaf vertex $u_{\sigma(k+\ell)}$. In this case, we say that $v_{\ell}$ is the \emph{parent vertex} of $u_{\sigma(k+\ell)}$ and that $u_{\sigma(k+\ell)}$ is a \emph{child vertex} of $v_{\ell}$.

\vspace{5mm}

In this way, one obtains a graph which is a disjoint union of $k$ binary trees. These binary trees are denoted by $\tau_j$, for $j=1,2,\ldots,k$. The root of $\tau_j$ is given by $w_j$. We note that the $\sigma_j$ in \eqref{product_structure} can be viewed as $\sigma \big|_{\tau_j}$. In other words, this is the internal labeling of the collision operators that respects the ordering of the vertices of $\tau_j$. Moreover, $\tau_j$ has $m_j$ internal vertices corresponding in this way to time labels $t_{\ell_{j,1}},t_{\ell_{j,2}}, \ldots, t_{\ell_{j,m_j}}$.

The internal vertex $v_r$ is called the \emph{distinguished vertex}. The other internal vertices are called \emph{regular}. The two children vertices corresponding to $v_r$ are called the \emph{distinguished leaves}. All the other leaves are called \emph{regular}. The tree $\tau_j$ is called \emph{distinguished} if $v_r \in \tau_j$ and otherwise it is called \emph{regular}. For more details on this construction, we refer the reader to \cite[Section 5]{ChHaPavSei}.

With this terminology in mind, we can rewrite the graph in our example as:

\vspace{5mm}

\begin{tikzpicture}[scale=0.2]
\draw node (w1) at (-5,5) {$w_1$};
\draw node (w2) at (-5,0) {$w_2$};
\draw node (w3) at (-5,-5) {$w_3$};
\draw node (B14) at (5,5) {$v_1$};
\draw node (B25) at (30,0) {$v_2$};
\draw node (B47) at (15,-10) {$v_4$};
\draw node (B36) at (55,-5) {$v_3$};
\draw node (B68) at (65,-20) {$v_5$};
\draw node (u1) at (75,5) {$u_1$};
\draw node (u2) at (75,0) {$u_2$};
\draw node (u3) at (75,-5) {$u_3$};
\draw node (u4) at (75,-10) {$u_4$};
\draw node (u5) at (75,-15) {$u_5$};
\draw node (u6) at (75,-20) {$u_6$};
\draw node (u7) at (75,-25) {$u_7$};
\draw node (u8) at (75,-30) {$u_8$};
\draw (B25)--(u5);
\draw (w1)--(B14);
\draw (B14)--(u1);
\draw (w2)--(B25);
\draw (B25)--(u2);
\draw [ultra thick] (w3)--(B36);
\draw [ultra thick] (B36)--(u3);
\draw (B14)--(B47);
\draw (B47)--(u4);
\draw [ultra thick] (B36)--(B68);
\draw [ultra thick] (B68)--(u6);
\draw (B47)--(u7);
\draw [ultra thick] (B68)--(u8);
\end{tikzpicture}

\vspace{5mm}

Namely, $v_1$ corresponds to $B_{1,4}$, $v_2$ corresponds to $B_{2,5}$, $v_3$ corresponds to $B_{3,6}$, $v_4$ corresponds to $B_{4,7}$, and $v_5$ corresponds to $B_{5,8}$. In this case, $v_5$ is the distinguished vertex. The distinguished tree is the tree with root $w_3$ and it is highlighted.

\subsubsection{\textbf{A detailed analysis of the distinguished tree graph}}

Let us now recall the analysis of the distinguished tree from \cite[Section 6]{ChHaPavSei}. As above, the analysis transfers to the periodic setting. For completeness, we will summarize the notation and the main ideas.

With notation as earlier, we suppose that $\tau_j$ is a distinguished tree. This tree contains $m_j$ internal vertices $(v_{\ell_j,a})_{a=1}^{m_j}$ and $m_j+1$ leaf vertices $(u_{j,s})_{s=1}^{m_j+1}$. The internal vertices are enumerated by $a \in \{1,2,\ldots,m_j\}$, and $a$ corresponds to the label in the collision operator $B_{\sigma_j(a+1),a+1}$. It is advantageous to continue this labeling and to denote the leaf vertices by $a \in \{m_j+1,\ldots,2m_j+1\}$. In other words, the leaf vertex $u_{j,a-m_j}$ corresponds to the label $a \in \{m_j+1,\ldots,2m_j+1\}$.

By the semigroup property of $\mathcal{U}^{(k)}(t)$, it is possible to show that the term in \eqref{product_structure} corresponding to the distinguished tree $\tau_j$ equals:

\begin{equation}
\notag
J^1_j(\phi;\sigma_j;t,t_{\ell_{j,1}},\ldots,t_{\ell_{j,m_j}};x_j;x'_j)=
\end{equation}
$$
=\mathcal{U}^{(1)}(t-t_{\ell_{j,1}})\,B_{1,2} \,\,\mathcal{U}^{(2)}(t_{\ell_{j,1}}-t_{\ell_{j,2}})\,B_{\sigma_j(3),3} \cdots B_{\sigma_j(a),a} \,\,\mathcal{U}^{(a)}(t_{\ell_{j,a-1}}-t_{\ell_{j,a}})\,B_{\sigma_j(a+1),a+1} \cdots
$$
$$\cdots \,\, \mathcal{U}^{(m_j)}(t_{\ell_{j,m_j-1}}-t_{\ell_{j,m_j}}) \,B_{\sigma_j(m_j+1),m_j+1}\,\big(|\phi \rangle \langle \phi|\big)^{\otimes (m_j+1)}.
$$

We note that, by construction, $\ell_{j,m_j}=r$. In other words, the last time at which we apply the Duhamel expansion is always $t_r$. This is because we are considering an $r$-fold Duhamel expansion in \eqref{Duhamel_Expansion_r} and since the distinguished vertex is the last internal vertex with respect to the given labeling.

We now want to bound the integral:

\begin{equation}
\label{IDT}
\int_{[0,T]^{m_j-1}} Tr(|J_j^1(\phi;\sigma_j;t,t_{\ell_{j,1}}, \ldots, t_{\ell_{j,m_j}})| \, d t_{\ell_{j,m_j-1}} \, dt_{\ell_{j,m_j-1}} \cdots dt_{\ell_{j,1}}.
\end{equation}
We note that, in \eqref{IDT}, one is not integrating with respect to $t_{\ell_{j,m_j}}=t_r$, i.e. to the time variable associated to the distinguished vertex. The reason is that one wants to use the integral in the variable $t_{r}$ to take into account the bound \eqref{Growth_Bound}. The details of this step are given in \eqref{gamma_k_bound_proof_1} and \eqref{gamma_k_bound_proof} below.

In order to estimate the expression in \eqref{IDT}, it is necessary to \emph{associate to each collision operator in $J^1_j$ an integral kernel}. In other words, one expands $J^1_j$ keeping in mind the graph structure. This is analogous to the calculations in the Example in Subsection \ref{Example} above. To the vertex labeled by $a$, one associates a kernel $\theta_a$. It can be seen that, at a regular leaf vertex $\theta_a(x;x'):=\phi(x) \cdot \overline{\phi(x')}$. Moreover, at the distinguished vertex, $\theta_{m_j}(x;x'):=\tilde{\psi}(x)\cdot \overline{\phi(x')} -\phi(x) \cdot \overline{\tilde{\psi}(x')},$ for $\tilde{\psi}:=|\phi|^2 \phi$.


Using an inductive formula, it is shown in \cite[Lemma 6.1]{ChHaPavSei} that for every $a \in \{1,2,\ldots,m_j\}$, the kernel $\theta_a$ can be written as:
\begin{equation}
\label{Kernel}
\theta_a(x;x')=\sum_{\beta_a} c_{\beta_a}^{a} \chi_{\beta_a}^{a}(x) \overline{\psi_{\beta_a}^{a}(x')}
\end{equation}
for at most $2^{m_j-\alpha+1}$ coefficients $c_{\beta_a}^{a} \in \{-1,1\}$.
In \cite[Section 6.4]{ChHaPavSei}, it is noted that the only dependence of $\theta_a$ on $t_{\ell_{j,a}}$ is by means of the propagators $\mathcal{U}_{a;\,\kappa_{+}(a)}:=e^{i(t_{\ell_{j,a}}-t_{\ell_{j,\kappa_{\pm}(a)}})\Delta}$. Moreover, the product $\chi_{\beta_a}^{a}(x) \overline{\psi_{\beta_a}^{a}(x')}$ is always of degree four and it either has the form:
\begin{equation}
\notag
\big(\,\mathcal{U}_{a;\,\kappa_{-}(a)}\, \chi_{\beta_{\kappa_{-}(a)}}^{\kappa_{-}(a)}(x)\big) \cdot \overline{\big(\,\mathcal{U}_{a,\,\kappa_{-}(a)} \,\psi_{\beta_{\kappa_{-}(a)}}^{\kappa_{-}(a)}\big)(x')} \cdot \big(\,\mathcal{U}_{a,\,\kappa_{+}(a)} \,\chi_{\beta_{\kappa_{+}(a)}}^{\kappa_{+}(a)}\big)(x) \cdot \overline{\big(\,\mathcal{U}_{a;\,\kappa_{+}(a)}\,\psi_{\beta_{\kappa_{+}(a)}^{\kappa_{+}(a)}}\big)(x)}
\end{equation}
or
\begin{equation}
\notag
\big(\,\mathcal{U}_{a;\,\kappa_{-}(a)}\, \chi_{\beta_{\kappa_{-}(a)}}^{\kappa_{-}(a)}(x)\big) \cdot \overline{\big(\,\mathcal{U}_{a,\,\kappa_{-}(a)} \,\psi_{\beta_{\kappa_{-}(a)}}^{\kappa_{-}(a)}\big)(x')} \cdot \big(\,\mathcal{U}_{a,\,\kappa_{+}(a)} \,\chi_{\beta_{\kappa_{+}(a)}}^{\kappa_{+}(a)}\big)(x') \cdot \overline{\big(\,\mathcal{U}_{a;\,\kappa_{+}(a)}\,\psi_{\beta_{\kappa_{+}(a)}^{\kappa_{+}(a)}}\big)(x')}.
\end{equation}
These identities 
allow one to inductively obtain the factors $\chi_{\beta_a}^{a}$ and $\psi_{\beta_a}^{a}$.


Finally, without loss of generality, it is possible to assume that $\psi^1_{\beta_1}$ and its iterates, i.e. the terms of the form $\psi_{\beta_{\kappa_{+}^{q}(1)}}^{\kappa_{+}^q(1)}$ are functions of $\tilde{\psi}=|\phi|^2 \phi$. We say that these factors are \emph{distinguished}. Here, $\kappa_{+}^q(1)$ denotes the $q-th$ iterate of $\kappa_{+}$ applied to the vertex with label $1$. By construction, this will mean that $\kappa_{+}^q(1)=m_j$ for some $q$. All of the arguments which we have presented here were originally noted in the non-periodic setting in \cite[Section 6]{ChHaPavSei}, but they carry over to the periodic setting without any change.

\subsubsection{\textbf{Recursive $L^2$ and $H^1$ bounds for the distinguished tree}}

We now summarize the main ideas from \cite[Section 7]{ChHaPavSei} in which one inductively uses $L^2$ and $H^1$ bounds in order to estimate the expression in \eqref{IDT}. As we will see, at this step, a modification of the argument will be required to deal with the periodic case. 

By the definition of the kernel $\theta_1$ and by \eqref{Kernel}, the expression in \eqref{IDT} is:

\begin{equation}
\notag
=\int_{[0,T]^{m_j-1}} Tr \big(|\,\mathcal{U}^{(1)}(t-t_{\ell_{j,1}}) \,\theta_1|\big) \, dt_{\ell_{j,m_j-1}} \,dt_{\ell_{j,m_j-2}} \cdots dt_{\ell_{j,1}}
\end{equation}
\begin{equation}
\label{IDT2} 
\leq \sum_{\beta_1} \int_{[0,T]^{m_j-1}} \|\psi_{\beta_1}^{1}\|_{L^2_x} \cdot \|\chi_{\beta_1}^{1}\|_{L^2_x} \, dt_{\ell_{j,m_j-1}} \,dt_{\ell_{j,m_j-2}} \cdots dt_{\ell_{j,1}}.
\end{equation}
In the last line, we used the unitarity of $\mathcal{U}^{(1)}$ and the Cauchy-Schwarz inequality in $\beta_1$.

In what follows, one denotes by $\tau_{j,a}$ the subtree of $\tau_j$ which is rooted at the vertex $a$. Given $\tau_{j,a}$, $d_a$ denotes the number of non-leaf (i.e. internal and root) vertices in $\tau_{j,a}$. Finally, one defines: 
$$\int_{[0,T]^{d_a}} (\cdots) \prod_{a' \in \tau_{j,a}} \, dt_{\ell_{j,a'}}$$ to denote the integral with respect to the time variables which correspond to the non-leaf vertices in $\tau_{j,a}$. If $\tau_{j,a}=\varnothing$, we just evaluate the integrand. If we emphasize that $a' \neq m_j$, then we will not be integrating with respect to $t_{\ell_{j,m_j}}$. With this notation in mind, the following estimates hold for all $a \in \{1,2,\ldots,m_j-1\}:$

\begin{equation}
\label{Bound_on_the_L2_level}
\int_{[0,T]^{d_a}} \|\psi_{\beta_a}^{a}\|_{L^2_x} \cdot \|\chi_{\beta_a}^{a}\|_{H^1_x} \prod_{a' \in \tau_{j,a}} \,dt_{\ell_{j,a'}}
\end{equation}
\begin{equation}
\notag
\leq CT^{\frac{1}{2}} \cdot \big( \int_{[0,T]^{d_{\kappa_{-}(a)}}} \|\psi_{\beta_{\kappa_{-}(a)}}^{\kappa_{-}(a)}\|_{H^1_x} \cdot \|\chi_{\beta_{\kappa_{-}(a)}}^{\kappa_{-}(a)}\|_{H^1_x} \prod_{a' \in \tau_{j,\kappa_{-}(a)}} \, dt_{\ell{j,a'}} \big)
\end{equation}
\begin{equation}
\notag
\cdot \,\big( \int_{[0,T]^{d_{\kappa_{+}(a)}}} \|\psi_{\beta_{\kappa_{+}(a)}}^{\kappa_{+}(a)}\|_{L^2_x} \cdot \|\chi_{\beta_{\kappa_{+}(a)}}^{\kappa_{+}(a)}\|_{H^1_x} \prod_{a' \in \tau_{j,\kappa_{+}(a)}} \, dt_{\ell{j,a'}} \big)
\end{equation}
and
\begin{equation}
\label{Bound_on_the_H1_level}
\int_{[0,T]^{d_a}} \|\psi_{\beta_a}^{a}\|_{H^1_x} \cdot \|\chi_{\beta_a}^{a}\|_{H^1_x} \prod_{a' \in \tau_{j,a}} \,dt_{\ell_{j,a'}}
\end{equation}
\begin{equation}
\notag
\leq CT^{\frac{1}{2}} \cdot \big( \int_{[0,T]^{d_{\kappa_{-}(a)}}} \|\psi_{\beta_{\kappa_{-}(a)}}^{\kappa_{-}(a)}\|_{H^1_x} \cdot \|\chi_{\beta_{\kappa_{-}(a)}}^{\kappa_{-}(a)}\|_{H^1_x} \prod_{a' \in \tau_{j,\kappa_{-}(a)}} \, dt_{\ell{j,a'}} \big)
\end{equation}
\begin{equation}
\notag
\cdot \,\big( \int_{[0,T]^{d_{\kappa_{+}(a)}}} \|\psi_{\beta_{\kappa_{+}(a)}}^{\kappa_{+}(a)}\|_{H^1_x} \cdot \|\chi_{\beta_{\kappa_{+}(a)}}^{\kappa_{+}(a)}\|_{H^1_x} \prod_{a' \in \tau_{j,\kappa_{+}(a)}} \, dt_{\ell{j,a'}} \big).
\end{equation}
The estimate in \eqref{Bound_on_the_L2_level} is referred to as the \emph{bound on the $L^2$ level}, and the estimate in \eqref{Bound_on_the_H1_level} is referred to as the \emph{bound on the $H^1$ level}.
The $L^2$ estimate \eqref{Bound_on_the_L2_level} is proved by using Proposition \ref{Multilinear_estimate} when $s=0$. In particular, we use:
\begin{equation}
\label{product_estimate}
\|\big(e^{it\Delta}f_1\big)(x) \cdot \overline{\big(e^{it\Delta}f_2\big)(x)} \cdot \big(e^{it\Delta}f_3(x)\big)\|_{L^2_{[0,T]}L^2_x} \leq C \|f_1\|_{H^1_x} \cdot \|f_2\|_{H^1_x} \cdot \|f_3\|_{L^2_x}.
\end{equation}
Here, we recall that, by assumption, $T \in [0,1]$.
The estimate \eqref{product_estimate} was also used in the non-periodic setting \cite[Section 7]{ChHaPavSei}. In this case, the trilinear estimate was deduced from Strichartz estimates on $\mathbb{R}^3$.
The rest of the proof of \eqref{Bound_on_the_L2_level} proceeds as in \cite[Section 7]{ChHaPavSei}. We will omit the details. The estimate \eqref{Bound_on_the_H1_level} is proved by using Proposition \ref{Multilinear_estimate} when $s=1$.

\subsubsection{\textbf{Estimating the factors corresponding to distinguished and regular trees}}
\label{Estimating the factors corresponding to distinguished and regular trees}

We will now record the bounds for the contributions from the distinguished tree and from the regular trees in \eqref{gamma_k_estimate}.

For the distinguished tree, it is necessary to estimate the expression in \eqref{IDT}. By applying \eqref{Bound_on_the_L2_level}, it is possible to bound the right-hand side of \eqref{IDT2} by \footnote{Strictly speaking, we are applying a variant of \eqref{Bound_on_the_L2_level} without the integration in $t_{\ell{j,m_j}}=t_r$. In other words, we are considering the subtree of $\tau_j$ obtained by deleting the distinguished vertex. For simplicity of exposition, we will not emphasize this distinction in what follows.}:

\begin{equation}
\label{IDT3}
\sum_{\beta_{\kappa_{-}(1)},\beta_{\kappa_{+}(1)}} CT^{\frac{1}{2}} \cdot \big(\int_{[0,T]^{d_{\kappa_{-}(1)}}} \,\|\psi_{\beta_{\kappa_{-}(1)}}^{\kappa_{-}(1)}\|_{H^1_x} \cdot \|\chi_{\beta_{\kappa_{-}(1)}}^{\kappa_{-}(1)}\|_{H^1_x} \,\prod_{a' \in \tau_{j,\kappa_{-}(1)}} \,dt_{a} \big)
\end{equation}
$$\cdot \big(\int_{[0,T]^{d_{\kappa_{+}(1)}-1}} \,\|\psi_{\beta_{\kappa_{+}(1)}}^{\kappa_{+}(1)}\|_{L^2_x} \cdot \|\chi_{\beta_{\kappa_{+}(1)}}^{\kappa_{+}(1)}\|_{H^1_x} \,\prod_{a' \in \tau_{j,\kappa_{+}(1)},\,a' \neq m_j}\,dt_{a'}\big).$$
The key point is that the distinguished factor $\psi_{\beta_{\kappa_{+}(1)}}^{\kappa_{+}(1)}$ is estimated in the $L^2$ norm. 
This is possible to do because in Proposition \ref{Multilinear_estimate}, when $s=0$, we can choose which factor we want to estimate in $L^2$.
By iterating \eqref{Bound_on_the_H1_level} as in \cite[Section 8]{ChHaPavSei}, it follows that:
\begin{equation}
\label{IDT4I}
\int_{[0,T]^{d_{\kappa_{-}(1)}}} \,\|\psi_{\beta_{\kappa_{-}(1)}}^{\kappa_{-}(1)}\|_{H^1_x} \cdot \|\chi_{\beta_{\kappa_{-}(1)}}^{\kappa_{-}(1)}\|_{H^1_x} \,\prod_{a' \in \tau_{j,\kappa_{-}(1)}} \,dt_{a}
\leq C_1^{d_{\kappa_{-}(1)}} \cdot T^{\frac{d_{\kappa_{-}(1)}}{2}} \cdot \|\phi\|_{H^1_x}^{2b_{\kappa_{+}(1)}}.
\end{equation}
Here, $b_a$ denotes the number of regular, i.e non-distinguished leaf vertices of the subtree $\tau_{j,a}$, which is rooted at $a$. $C_1>0$ denotes the constant in \eqref{Bound_on_the_H1_level}.

Moreover, by combining \eqref{Bound_on_the_L2_level} and \eqref{Bound_on_the_H1_level} as in \cite[Section 8]{ChHaPavSei}, it follows that:

\begin{equation}
\label{IDT4IIa}
\int_{[0,T]^{d_{\kappa_{+}(1)}-1}} \,\|\psi_{\beta_{\kappa_{+}(1)}}^{\kappa_{+}(1)}\|_{L^2_x} \cdot \|\chi_{\beta_{\kappa_{+}(1)}}^{\kappa_{+}(1)}\|_{H^1_x} \,\prod_{a' \in \tau_{j,\kappa_{+}(1)},\,a' \neq m_j}\,dt_{a'}
\end{equation} 
$$\leq C_1^{d_{\kappa_{+}(1)}-1} \cdot T^{\frac{d_{\kappa_{+}(1)}-1}{2}} \cdot \|\phi\|_{H^1_x}^{2b_{\kappa_{+}(1)}+1} \cdot \||\phi|^2 \phi\|_{L^2_x}.$$

Here, we take $C_1$ to be a constant which is possibly larger than above in order to take into account the implied constant in \eqref{Bound_on_the_L2_level}.

By applying the Sobolev embedding result in \eqref{Sobolev_embedding_torus}, it follows that the right-hand side of \eqref{IDT4IIa} is:

\begin{equation}
\label{IDT4II}
\leq C_2^{d_{\kappa_{+}(1)-1}} \cdot T^{\frac{d_{\kappa_{+}(1)-1}}{2}} \cdot \|\phi\|_{H^1_x}^{2b_{\kappa_{+}(1)}+4},
\end{equation}
for some constant $C_2>0$. In particular, by \eqref{IDT4I} and \eqref{IDT4II}, the sum in \eqref{IDT3} is:
$$\leq \sum_{\beta_{\kappa_{+}(1)},\beta_{\kappa_{-}(1)}} C_3^{d_{\kappa_{-}(1)}+d_{\kappa_{+}(1)}} \cdot T^{\frac{d_{\kappa_{+}(1)}+d_{\kappa_{-}(1)}}{2}} \cdot \|\phi\|_{H^1_x}^{2b_{\kappa_{-}(1)}+2b_{\kappa_{+}(1)}+4},$$
for some constant $C_3>0$.
Finally, we use the fact that:
\begin{equation}
\notag
\begin{cases}
d_{\kappa_{-}(1)}+d_{\kappa_{+}(1)}=m_j-1\\
b_{\kappa_{-}(1)}+b_{\kappa_{+}(1)}=m_j-1.
\end{cases}
\end{equation}
as well as the expansion \eqref{Kernel} to count the number of terms in the sum in order to deduce that, for some constant $C_4>0$:

\begin{equation}
\label{IDT5}
\int_{[0,T]^{m_j-1}} Tr \big(|J_j^1(\phi;\sigma_j;t,t_{\ell_{j,1}},\ldots,t_{\ell_{j,m_j}})|\big) \, dt_{\ell_{j,m_j-1}} \cdots \, dt_{\ell_{j,1}}
\end{equation} 
\begin{equation}
\notag
\leq C_4^{m_j} \cdot T^{\frac{m_j-1}{2}} \cdot \|\phi\|_{H^1_x}^{2m_j+2}.
\end{equation}
This is the bound that one uses for the distinguished tree $\tau_j$.

Suppose now that $\tau_j$ is a regular tree. In this case, we can iteratively use the estimate \eqref{Bound_on_the_H1_level} and obtain the bound as in \cite[Proposition 8.2]{ChHaPavSei}:

\begin{equation}
\label{IRTB}
\int_{[0,T]^{m_j}} Tr\big(|J_j^1(\phi;\sigma_j;t,t_{\ell{j,1}},\ldots,t_{\ell{j,m_j}})|\big) \, dt_{\ell{j,m_j}} \,dt_{\ell{j,m_j-1}} \cdots \,dt_{\ell{j,1}}
\end{equation}
\begin{equation}
\notag
\leq C_4^{m_j} \cdot T^{\frac{m_j}{2}} \cdot \|\phi\|_{H^1_x}^{2m_j+2}.
\end{equation}
Here, we choose a possibly larger constant $C_4>0$. This is the bound that we use for the regular trees. 

\subsubsection{\textbf{The proof of Theorem \ref{Unconditional_uniqueness}}}

By using the previous results, we will now give the proof of Theorem \ref{Unconditional_uniqueness}:

\begin{proof}
Combining \eqref{gamma_k_estimate}, \eqref{IDT5}, and \eqref{IRTB}, it follows that for all $k \in \mathbb{N}$ and $t \in [0,T]$:
\begin{equation}
\label{gamma_k_bound_proof_1}
Tr(|\gamma^{(k)}(t)|) \leq C^{k+r} \cdot \sum_{\ell=1}^{2} C_4^r \cdot T^{\frac{r-1}{2}} \cdot \int_{0}^{t} \int_{\mathcal{B}} \|\phi\|_{H^1_x}^{2(k+r)}\,d\mu^{(\ell)}_{t_r} (\phi) \, dt_r.
\end{equation}
By \eqref{Growth_Bound}, this expression is:
\begin{equation}
\label{gamma_k_bound_proof}
\leq 2 \cdot C^{k+r} \cdot C_4^r \cdot T^{\frac{r-1}{2}} \cdot \int_{0}^{t} M^{k+r} \, dt_r \leq 2 \cdot C^{k+r} \cdot C_4^r \cdot T^{\frac{r+1}{2}} \cdot M^{k+r}.
\end{equation}
Here $M:=\max\{M_1,M_2\}$, where $M_1$ and $M_2$ are as in \eqref{Growth_Bound}. The upper bound in \eqref{gamma_k_bound_proof} equals:
\begin{equation}
\notag
2 \cdot C^k \cdot M^k \cdot T^{\frac{1}{2}} \cdot (C \cdot C_4 \cdot M \cdot T^{\frac{1}{2}})^r
\end{equation}
We first suppose that $T \in (0,1]$ is sufficiently small such that $C \cdot C_4 \cdot M \cdot T^{\frac{1}{2}} <1$ and we let $r \rightarrow \infty$ to deduce that $\gamma^{(k)} \equiv 0$ on $[0,T]$. For general $T$, we divide the time interval $[0,T]$ into finitely many pieces which are sufficiently small and we repeat the argument. This proves \eqref{claim} and the unconditional uniqueness result now follows.
In order to deduce \eqref{gammak_formula}, we note that:
\begin{equation}
\notag
\tilde{\gamma}^{(k)}(t):=\int_{\mathcal{B}} \big(|S_t(\phi) \rangle \langle S_t(\phi)|^{\otimes k}\big) \,d\mu(\phi).
\end{equation}
gives a mild solution to \eqref{GPhierarchy} in $L^{\infty}_{t \in [0,T]} \mathfrak{H}^1$ with initial data
\begin{equation}
\notag
\tilde{\gamma}^{(k)}(0)=\int_{\mathcal{B}} \big(|\phi \rangle \langle \phi|^{\otimes k}\big) \,d\mu(\phi)=\gamma^{(k)}(0).
\end{equation}
We apply the above uniqueness result in order to deduce that, for all $k \in \mathbb{N}$ and $t \in [0,T]$:
\begin{equation}
\notag
\tilde{\gamma}^{(k)}(t)=\gamma^{(k)}(t).
\end{equation}
The theorem now follows.
\end{proof}


\section{The derivation of the defocusing cubic NLS on $\mathbb{T}^3$}
\label{derivation of defocusing cubic NLS on T3}

We will now apply the above unconditional uniqueness theorem in order to obtain a derivation of the defocusing cubic NLS on $\mathbb{T}^3$, following the program of Elgart, Erd\H{o}s, Schlein, and Yau \cite{EESY} and Erd\H{o}s, Schlein, and Yau \cite{ESY1,ESY2,ESY3,ESY4,ESY5}. 
Let us note that the results of Section \ref{An unconditional uniqueness result} apply both in the focusing and in the defocusing regime. The results of this section will apply only in the \emph{defocusing} regime.
More precisely, in what follows, we will set $\lambda=1$ in \eqref{GPhierarchy} and \eqref{NLS_T3}. As before, $S_t$ will denote the corresponding flow map. The main point that we want to address in the analysis of this section is that, for appropriately chosen initial data, the solution to the Gross-Pitaevskii hierarchy obtained as a limit in \cite{EESY} satisfies the assumptions of Theorem \ref{Unconditional_uniqueness}. In this context, it is possible to apply the unconditional uniqueness result from Theorem \ref{Unconditional_uniqueness} and uniquely characterize the limit 
point $\Gamma_{\infty,t}$. This will be an important idea in the proof of the main result, which is stated as Theorem \ref{Cubic_NLS_T3_derivation} below. 

For completeness, let us now briefly summarize the main results of \cite{EESY}. Suppose that $V: \mathbb{R}^3 \rightarrow \mathbb{R}$ is a smooth, non-negative function with compact support such that $\int_{\mathbb{R}^3}V(x)\,dx=1$. The bosons in $\mathbb{T}^3$ interact via a two-body potential which is given by: 
\begin{equation}
\notag
V_N(x):=N^{3\beta} V(N^{\beta}x),
\end{equation} 
for $\beta>0$ a parameter. The $N-$body Hamiltonian for the $N$ weakly coupled bosons is given by:
\begin{equation}
\label{HN}
H_N:=-\sum_{j=1}^{N}\Delta_j+\frac{1}{N}\sum_{\ell < j}^{N}V_N(x_{\ell}-x_j)
\end{equation}
In \cite{EESY}, it is assumed that:
\begin{equation}
\label{assumption_on_beta}
\beta \in (0,\frac{3}{5}).
\end{equation}
In what follows, we will make the same assumption on $\beta$.
$H_N$ acts on a dense subset of $L^2_{sym}(\Lambda^N \times \Lambda^N)$.
We remark that $V_N$ is originally a function on $\mathbb{R}^3$. Nevertheless, for $N$ large, $V_{N}$ is supported in a small neighborhood of $0$. Hence, it can be extended to a periodic function on $\mathbb{R}^3$. By this procedure, it can be thought of as a function on $\mathbb{T}^3$. We will use this convention in our paper.

Let us fix $\Psi_{N,0} \in L^2_{sym}(\Lambda^N)$. Let $\Psi_{N,t}$ be the solution of the $N$-body Schr\"{o}dinger equation associated to $H_N$ with initial data $\Psi_{N,0}$ given in \eqref{PsiNt_equation}.
We recall from \eqref{gammakNt_original} that if we take
$\gamma^{(k)}_{N,t}:=Tr_{k+1,\ldots,N} \,\big|\Psi_{N,t} \rangle \langle \Psi_{N,t}\big|,$ the sequence $(\gamma^{(k)}_{N,t})_k$ solves the BBGKY hierarchy \eqref{BBGKY}.


In \cite{EESY,ESY1}, the authors define for $\Gamma \in \bigoplus_{k \in \mathbb{N}}L^2_{sym}(\Lambda^k \times \Lambda^k)$
the following quantities:
\begin{equation}
\label{H-}
\|\Gamma\|_{H_{-}}:=\sum_{k=1}^{+\infty} 2^{-k} \cdot \|\gamma^{(k)}\|_{L^2(\Lambda^k \times \Lambda^k)}^2
\end{equation}
\begin{equation}
\label{H+}
\|\Gamma\|_{H_{+}}:=\sup_{k \geq 1} 2^k \cdot \|\gamma^{(k)}\|_{L^2(\Lambda^k \times \Lambda^k)}^2.
\end{equation}
One then defines $H_{-}$ to be the set of $\Gamma \in \bigoplus_{k \in \mathbb{N}}L^2_{sym}(\Lambda^k \times \Lambda^k)$ for which $\|\Gamma\|_{H^{-}}$ is finite. The space $H_{+}$ is defined as $\{\Gamma \in \bigoplus_{k \in \mathbb{N}} L^2(\Lambda^k \times \Lambda^k),\,\lim_{k \rightarrow \infty} 2^k \cdot \|\gamma^{(k)}\|_{L^2(\Lambda^k \times \Lambda^k)}^2=0\}$, with the norm given by \eqref{H+}. $H_{-}$ and $H_{+}$ are then Banach spaces. Furthermore: $$(H_{+})^{*}=H_{-}$$
with respect to the pairing:
$$(\Gamma_1,\Gamma_2) \mapsto \sum_{k=1}^{+\infty} \langle \gamma_1^{(k)},\gamma_2^{(k)} \rangle_{L^2(\Lambda^k \times \Lambda^k)},$$
whenever $\Gamma_1=(\gamma_1^{(k)})_k$ and $\Gamma_2=(\gamma_2^{(k)})_k$ in $\bigoplus_{k \in \mathbb{N}}L^2_{sym}(\Lambda^k \times \Lambda^k)$.

Given $T>0$, $C([0,T],H_{-})$ denotes the space of functions of $t \in [0,T]$ which take values in $H_{-}$, and which are continuous with respect to the weak$-*$ topology on $H_{-}$. The space $H_{+}$ is separable, and hence it is possible to find a countable dense subset in the unit ball of $H_{+}$, which is denoted by $\{\mathcal{J}_{i,\ell}, \ell \in \mathbb{N}\}$. The metric $\varrho$ on $H_{-}$ is defined as:
\begin{equation}
\label{rho}
\varrho\,(\Gamma_1,\Gamma_2):=
\end{equation}
\begin{equation}
\notag
\sum_{\ell=1}^{+\infty} 2^{-\ell} \cdot \Big|\sum_{k=1}^{+\infty} \int_{\Lambda^k \times \Lambda^k} \overline{\mathcal{J}_{i,\ell}^{(k)}(\vec{x}_k;\vec{x}'_k)} \cdot \big[\gamma_1^{(k)}(\vec{x}_k;\vec{x}'_k)-\gamma_2^{(k)}(\vec{x}_k;\vec{x}'_k)]\big] \, d\vec{x}_k\,d\vec{x}'_k \Big|,
\end{equation}
whenever $\Gamma_1=(\gamma^{(k)}_1)_k, \Gamma_2=(\gamma^{(k)}_2)_k$ in $H_{-}$.

The metric $\widehat{\varrho}$ on $C([0,T],H_{-})$ is defined as:
\begin{equation}
\notag
\widehat{\varrho}\,(\Gamma_1,\Gamma_2):=\sup_{t \in [0,T]} \varrho\,(\Gamma_1(t),\Gamma_2(t)),
\end{equation}
whenever $\Gamma_1,\Gamma_2 \in C([0,T],H_{-})$. For more details on these spaces, we refer the reader to \cite{EESY,ESY1}.

The main result of \cite{EESY} is the following:

\begin{theorem}(Theorem 1 in \cite{EESY})
\label{EESY_Theorem}
Given $N \in \mathbb{N}$, let $H_N$ be as in \eqref{HN}. Suppose that $\gamma_{N,0} \in L^2_{sym}(\Lambda^N \times \Lambda^N)$ is such that $\gamma_{N,0}$ is self-adjoint, $\gamma_{N,0} \geq 0$ as an operator on $L^2(\Lambda^N)$ and $Tr \,\gamma_{N,0}=1$. Furthermore, suppose that there exists a constant $C_1>0$ such that for all $k,N \in \mathbb{N}$:
\begin{equation}
\label{Assumption_EESY}
Tr\, H_N^k \gamma_{N,0} \leq C_1^k N^k.
\end{equation} 
Let us fix $T>0$ and let $\Gamma_{N,t}=(\gamma^{(k)}_{N,t})_k$ be the solution to the BBGKY hierarchy \eqref{BBGKY} with initial data $\Gamma_{N,0}=(\gamma^{(k)}_{N,0})_k$.
Then:
\begin{itemize}
\item[i)] The sequence $(\Gamma_{N,t})_N$ is relatively compact in $C([0,T],H_{-})$ with respect to the metric $\widehat{\varrho}$.
\item[ii)] If $\Gamma_{\infty,t}=(\gamma^{(k)}_{\infty,t})_k$ is a limit point of $(\Gamma_{N,t})_N$ with respect to the metric $\widehat{\varrho}$, then there exists a constant $C_2>0$ such that:
\begin{equation}
\notag
Tr \big|S^{(k,1)}\gamma^{(k)}_{\infty,t}\big| \leq C_2^k,
\end{equation}
for all $k \in \mathbb{N}$ and for all $t \in [0,T]$.
\item[iii)] $Tr \, \gamma^{(k)}_{\infty,t}=1$ for all $k \in \mathbb{N}$ and for all $t \in [0,T]$.
\item[iv)] $\Gamma_{\infty,t}$ solves the Gross-Pitaevskii hierarchy \eqref{GPhierarchy} with $\lambda=1$.
\end{itemize}
\end{theorem}

\begin{remark}
In the assumption \eqref{Assumption_EESY}, we note that the operator $H_N$ acts only in the first $N$ components.
\end{remark}

\begin{remark}
$\Gamma_{\infty,t}$ constructed in Theorem \ref{EESY_Theorem} is a mild solution of the Gross-Pitaevskii hierarchy in the sense of Definition \ref{Mild_solution}.
The precise argument showing that the solution of type as is constructed in \cite{EESY} is a mild solution is given in \cite[equation (8.31)]{ESY2}.
\end{remark}

\begin{remark}
The fact that we can take $\lambda=1$ in \eqref{GPhierarchy} follows from the assumption that $\int_{\mathbb{R}^3} V(x)\,dx=1$.
\end{remark}

In the discussion that follows, let us fix $\phi \in H^1(\Lambda)$ with $\|\phi\|_{L^2(\Lambda)}=1$. By properly choosing the initial data, we will show that Theorem \ref{Unconditional_uniqueness} implies that $\Gamma_{\infty,t}$ is uniquely determined and that $\gamma^{(k)}_{\infty,t}=|S_t(\phi)\rangle \langle S_t(\phi)|^{\otimes k}$. 
Let us observe that, by the assumption \eqref{Assumption_EESY} of Theorem \ref{EESY_Theorem}, it is not possible to directly take factorized initial data $\gamma_{N,0}=|\phi^{\otimes N} \rangle \langle \phi^{\otimes N}|=|\phi \rangle \langle \phi|^{\otimes N}$. Namely, then \eqref{Assumption_EESY} only holds when $k=1$. In order to obtain the full range of $k$, it is necessary to perform an additional approximation argument which is outlined in Subsection \ref{Approximation} below.

Let us consider a sequence $(\Psi_N)_N \in \bigoplus_{N \in \mathbb{N}} L^2_{sym}(\Lambda^N)$ which satisfies the following assumptions:

1)\,\emph{Bounded energy per particle:}
\begin{equation}
\label{Bounded_energy_per_particle}
\mathop{sup}_{N} \frac{1}{N} \langle H_N \Psi_N, \Psi_N \rangle_{L^2(\Lambda^N)} < +\infty.
\end{equation}
Here, $H_N$ is the $N$-body Hamiltonian defined in \eqref{HN}.
2)\,\emph{Asymptotic factorization:}
\begin{equation}
\label{Asymptotic_factorization}
Tr \big|\gamma^{(1)}_{N}-|\phi \rangle \langle \phi| \big| \rightarrow 0
\end{equation}
as $N \rightarrow \infty$. Here, $\gamma^{(1)}_N:=Tr_{2,3,\ldots,N} |\Psi_N \rangle \langle \Psi_N|$.

The main result of this section is the following:

\begin{theorem}{(A derivation of the defocusing cubic NLS on $\mathbb{T}^3$)}
\label{Cubic_NLS_T3_derivation} 

Let $\phi \in H^1(\Lambda)$ with $\|\phi\|_{L^2(\Lambda)}=1$ be given.
Suppose that the sequence $(\Psi_N)_N \in \bigoplus_{N \in \mathbb{N}} L^2(\Lambda^N)$ satisfies the assumptions \eqref{Bounded_energy_per_particle} and \eqref{Asymptotic_factorization}. For fixed $N \in \mathbb{N}$, we define: $\gamma^{(k)}_N:=Tr_{k+1,\ldots,N}|\Psi_N \rangle \langle \Psi_N|$. Let $(\gamma^{(k)}_{N,t})_k$ denote the solution of the BBGKY hierarchy evolving from initial data $(\gamma^{(k)}_N)_k$.

Let us fix $T>0$. There exists a sequence $N_j \rightarrow \infty$ as $j \rightarrow \infty$ such that, for all $k \in \mathbb{N}$ and for all $t \in [0,T]$:

\begin{equation}
\label{Convergence_result}
Tr \,\Big|\gamma^{(k)}_{N_j,t}-|S(t)\phi \rangle \langle S(t)\phi|^{\otimes k}\Big| \rightarrow 0
\end{equation}
as $j \rightarrow \infty$.
\end{theorem}


In particular, we obtain:



\begin{corollary}{(Evolution of purely factorized states)}
\label{Evolution_of_purely_factorized_states}

In the assumptions \eqref{Bounded_energy_per_particle} and \eqref{Asymptotic_factorization}, it is possible to take purely factorized states $\Psi_N=\phi^{\otimes N}$ for $\phi \in H^1(\Lambda)$ with $\|\phi\|_{L^2(\Lambda)}=1$. 
Thus, \eqref{Convergence_result} holds if the initial data is purely factorized.
\end{corollary}

\subsection{Approximation of the initial data}
\label{Approximation}
We would like to use Theorem \ref{EESY_Theorem}  to prove Theorem \ref{Cubic_NLS_T3_derivation}. In order to do this, we will first apply an approximation procedure to the sequence $(\Psi_N)_N$ satisfying \eqref{Bounded_energy_per_particle} and \eqref{Asymptotic_factorization} so that the initial data satisfies the assumption \eqref{Assumption_EESY} of Theorem \ref{EESY_Theorem}.
Let us choose $\zeta \in C_0^{\infty}(\mathbb{R})$ with $0 \leq \zeta \leq 1$, $\zeta \equiv 1$ on $[0,1]$ and $\zeta \equiv 0$ on $(2,+\infty)$. Finally, let us fix $\kappa>0$ to be a small parameter. 

As in \cite[Proposition 5.1]{ESY2}, \cite[Proposition 9.1]{ESY4}, and \cite[Proposition 8.1]{ESY5}, we define the following approximation to $\Psi_N$:

\begin{equation}
\notag
\widetilde{\Psi}_N^{\kappa}:=\frac{\zeta(\kappa \cdot \frac{H_N}{N})\Psi_N}{\,\,\,\,\,\,\|\zeta(\kappa \cdot \frac{H_N}{N})\Psi_N\|_{L^2(\Lambda^N)}}.
\end{equation}
In other words, we smoothly cut off the high frequencies in $\Psi_N$, and we normalize the result in $L^2(\Lambda^N)$.

The function $\widetilde{\Psi}_{N}^{\kappa} \in L^2_{sym}(\Lambda^N)$ then satisfies the following properties:

\medskip

$1)$ $\langle H_{N}^k\, \widetilde{\Psi}_{N}^{\kappa}, \widetilde{\Psi}_{N}^{\kappa} \rangle = Tr\,H_N^k \,\big|\widetilde{\Psi}_{N}^{\kappa} \rangle \langle \widetilde{\Psi}_{N}^{\kappa}\big| \leq \frac{2^k N^k}{\kappa^k}.$

\smallskip

$2)$ For some universal constant $C>0$:
\begin{equation}
\label{Approximation2}
\mathop{sup}_{N} \|\Psi_N-\widetilde{\Psi}_{N}^{\kappa}\|_{L^2(\Lambda^N)} \leq C \kappa^{\frac{1}{2}}.
\end{equation}

$3)$ We let $\widetilde{\gamma}^{(k)}_{\kappa;\,N}:=Tr_{k+1,\ldots,N} \big|\widetilde{\Psi}_{N}^{\kappa} \rangle \langle \widetilde{\Psi}_{N}^{\kappa}\big|$. For $\kappa>0$ small enough, it is then the case that:
\begin{equation}
\label{Approximation3}
\widetilde{\gamma}^{(k)}_{\kappa;\,N} \rightharpoonup^{*} |\phi \rangle \langle \phi|^{\otimes k}
\end{equation}
as $N \rightarrow \infty$ in the weak-$*$ topology on $\mathcal{L}^1_k$.

\vspace{5mm}
This approximation procedure was first used in the non-periodic setting \cite{ESY2,ESY4,ESY5}. The proofs of the analogues of points $1)$ and $2)$ in the non-periodic setting \cite[Proposition 5.1 $i), ii)$]{ESY2} and \cite[Proposition 8.1 $i),ii)$]{ESY5} carry over to the periodic setting. We note that the proof of the analogue of $1)$, given in \cite{ESY2}, relies on a Sobolev type inequality whose analogoue on $\mathbb{T}^3$ was proven in Appendix A of \cite{EESY}. The proof of the non-periodic analogue of point $3)$, given in \cite[Proposition 5.1 $iii)$]{ESY2}, which builds on the previous work of Michelangeli \cite{Michelangeli}, also carries over to the periodic setting by using an appropriate extension argument. In \cite{ESY2}, it is assumed that $\beta<\frac{2}{3}$, which holds since we are considering $\beta \in (0,\frac{3}{5})$ by \eqref{assumption_on_beta}. A proof of the non-periodic variant of $3)$ is also given in \cite[Proposition 8.1 $iii)$]{ESY5}. We will omit the details. Let us note that the same approximation argument was also used in the $2D$ periodic setting in \cite{KSS}. 

\medskip

We observe that, by $1)$, $\big|\widetilde{\Psi}_{N}^{\kappa} \rangle \langle \widetilde{\Psi}_{N}^{\kappa}\big|$ satisfies the assumption \eqref{Assumption_EESY} with implied constant $\frac{2}{\kappa}$. Let us fix $T>0$ and let $\widetilde{\Gamma}_{\kappa;\,N,t}=(\widetilde{\gamma}^{(k)}_{\kappa;\,N,t})_k$ be the solution to the BBGKY hierarchy \eqref{BBGKY} with initial data $\widetilde{\Gamma}_{\kappa;\,N,0}=(\widetilde{\gamma}^{(k)}_{\kappa;\,N})_k$ on the time interval $[0,T]$. In particular, for all $k \in \mathbb{N}$ and for all $t \in [0,T]$: \begin{equation}
\label{gamma_tilde_kappa_N_t}
\widetilde{\gamma}^{(k)}_{\kappa;\,N,t}=Tr_{k+1,\ldots,N}\,\big|\widetilde{\Psi}_{N,t}^{\kappa} \rangle \langle \widetilde{\Psi}_{N,t}^{\kappa}\big|, 
\end{equation}
where $\widetilde{\Psi}_{N,t}^{\kappa}$ is the solution of:
\begin{equation}
\notag
\begin{cases}
i\partial_t \widetilde{\Psi}_{N,t}^{\kappa}=H_N \widetilde{\Psi}_{N,t}^{\kappa}\\
\widetilde{\Psi}_{N,t}^{\kappa}\,\big|_{t=0}=\widetilde{\Psi}_{N}^{\kappa}.
\end{cases}
\end{equation}
By Theorem \ref{EESY_Theorem}, there exists a limit point $\widetilde{\Gamma}_{\kappa;\,\infty,t}=(\widetilde{\gamma}^{(k)}_{\kappa;\,\infty,t})_k$ of $(\widetilde{\Gamma}_{\kappa;\,N,t})_N$ with respect to the metric $\widehat{\varrho}$ which solves the GP hierarchy. In other words, there exists $N_j \rightarrow  \infty$ such that:
\begin{equation}
\label{Nj_infty1}
\widetilde{\Gamma}_{\kappa;\,N_j,t} \rightarrow \widetilde{\Gamma}_{\kappa;\,\infty,t} 
\end{equation}
as $j \rightarrow \infty$ with respect to the metric $\widehat{\varrho}$.


Let us note that, for fixed $k \in \mathbb{N}$ and $t \in [0,T]$:

\begin{equation}
\label{trace_gamma_equal_to_1}
Tr\, \widetilde{\gamma}^{(k)}_{\kappa;\,N,t}= Tr\,\big|\widetilde{\Psi}_{N,t}^{\kappa} \rangle \langle \widetilde{\Psi}_{N,t}^{\kappa} \big|= Tr\,\big|\widetilde{\Psi}_{N}^{\kappa} \rangle \langle \widetilde{\Psi}_{N}^{\kappa} \big|= \big\|\widetilde{\Psi}_{N}^{\kappa} \big\|_{L^2(\Lambda^N)}^2=1.
\end{equation}
Here, we used the fact that $Tr\,\big|\widetilde{\Psi}_{N,t}^{\kappa} \rangle \langle \widetilde{\Psi}_{N,t}^{\kappa} \big|=\big\|\widetilde{\Psi}^{\kappa}_{N,t}\big\|_{L^2(\Lambda^N)}^2$ is conserved in time.
Furthermore, we note that, for all $(\vec{x}_k;\vec{x}'_k) \in \Lambda^k \times \Lambda^k$:
\begin{equation}
\notag
\widetilde{\gamma}^{(k)}_{\kappa;\,N,t}(\vec{x}_k;\vec{x}'_k)=\int_{\Lambda^{N-k}} \widetilde{\Psi}^{\kappa}_{N,t}(\vec{x}_k,\vec{y}_{N-k}) \cdot \overline{\widetilde{\Psi}^{\kappa}_{N,t}(\vec{x}'_k,\vec{y}_{N-k})}\,d\vec{y}_{N-k}.
\end{equation}
Hence, for all $\vec{x}_k \in \Lambda^k$: 
\begin{equation}
\notag
\big|\widetilde{\gamma}^{(k)}_{\kappa;\,N,t}(\vec{x}_k;\vec{x}_k)\big|=\widetilde{\gamma}^{(k)}_{\kappa;\,N,t}(\vec{x}_k;\vec{x}_k)=\int_{\Lambda^{N-k}} \big|\widetilde{\Psi}^{\kappa}_{N,t}(\vec{x}_k,\vec{y}_{N-k})\big|^2 \,d\vec{y}_{N-k}.
\end{equation}
In particular:
\begin{equation}
\label{trace_gamma_tilde_equal_to_1}
Tr\,\big|\widetilde{\gamma}^{(k)}_{\kappa;\,N,t}\big|=Tr\,\widetilde{\gamma}^{(k)}_{\kappa;\,N,t}=1.
\end{equation}
\subsection{A comparison of different forms of convergence}
\label{A comparison of different forms of convergence}

Let us henceforth fix $\kappa$ to be sufficiently small so that \eqref{Approximation3} holds.
We would like to use Theorem \ref{Unconditional_uniqueness} and deduce that $\widetilde{\gamma}^{(k)}_{\kappa;\,\infty,t}=\big|S_t(\phi) \rangle \langle S_t(\phi)\big|^{\otimes k}$ for all $k \in \mathbb{N}$ and for all $t \in [0,T]$. 
We can not directly apply Theorem \ref{Unconditional_uniqueness} since the convergence in \eqref{Nj_infty1} is given in terms of the metric $\widehat{\varrho}$, and in the assumptions of the theorem, the convergence of each component is given with respect to the weak-$*$ topology on $\mathcal{L}^1_k$.

The arguments in \cite[Section 4.1]{EESY} imply that the sequence $\big(\widetilde{\Gamma}_{\kappa;\,N,t}\big)_N$ satisfies the following equicontinuity result:


For all $k \geq 1$ and for all $\mathcal{Z}^{(k)} \in W^{1,\infty}(\Lambda^k \times \Lambda^k)$, it is the case that, for every $\epsilon>0$, there exists $\delta>0$ such that:

\begin{equation}
\label{equicontinuity1}
\Big|\big\langle \mathcal{Z}^{(k)}, \widetilde{\gamma}^{(k)}_{\kappa;\,N,t}-\widetilde{\gamma}^{(k)}_{\kappa;\,N,s}\big\rangle_{L^2(\Lambda^k \times \Lambda^k)}\Big| \leq \epsilon,
\end{equation}
for all $s,t \in [0,T]$ with $|s-t| \leq \delta.$ The quantity $\delta$ depends on $k$ and $\mathcal{Z}^{(k)}$, but is independent of $N$. 

We would like to prove an equicontinuity result of the type \eqref{equicontinuity1} in the weak-$*$ topology on $\mathcal{L}^1_k$. In order to do this, we  first prove the following:

\begin{lemma}
\label{Density}
Let us fix $k \in \mathbb{N}$. Then, integral operators on $L^2(\Lambda^k)$ with kernels in $W^{1,\infty}(\Lambda^k \times \Lambda^k)$ are dense in $\mathcal{K}_k$ with respect to the operator norm topology.
\end{lemma}

\begin{proof}
Suppose that $\mathcal{J}^{(k)} \in \mathcal{K}_k$. Let $\epsilon>0$ be given. We can then find a finite-rank operator $T_1$ such that $\|\mathcal{J}^{(k)}-T_1\|<\frac{\epsilon}{2}$. The operator $T_1$ will have a kernel of the form: $K_1(\vec{x}_k;\vec{x}'_k)=\sum_{j=1}^{M} \alpha_j \cdot f_j(\vec{x}_k) \cdot g_j(\vec{x}'_k)$ for some $M \in \mathbb{N}, \alpha_1,\alpha_2, \ldots, \alpha_M \in \mathbb{C}$, and for some $f_1,\ldots,f_M,g_1,\ldots,g_M \in L^2(\Lambda^k)$. In particular, $K_1 \in L^2(\Lambda^k \times \Lambda^k)$. Since $W^{1,\infty}(\Lambda^k \times \Lambda^k)$ is dense in $L^2(\Lambda^k \times \Lambda^k)$ with respect to $\|\cdot\|_{L^2(\Lambda^k \times \Lambda^k)}$ (we can see this by truncating Fourier series), it follows that there exists $\mathcal{Z}^{(k)} \in W^{1,\infty}(\Lambda^k \times \Lambda^k)$ such that $\|\mathcal{Z}^{(k)}-K_1\|_{L^2(\Lambda^k \times \Lambda^k)}<\frac{\epsilon}{2}$. We denote the integral operator associated with $\mathcal{Z}^{(k)}$ by $T$. It follows that $\|T-T_1\|=\|\mathcal{Z}^{(k)}-K_1\|_{L^2(\Lambda^k \times \Lambda^k)}< \frac{\epsilon}{2}$ and hence $\|\mathcal{J}^{(k)}-T\| \leq \|\mathcal{J}^{(k)}-T_1\|+\|T-T_1\|<\epsilon.$ The density result now follows.
\end{proof}

Let us return to the proof of the equicontinuity result in the weak-$*$ topology on $\mathcal{L}^1_k$. Suppose that $\mathcal{J}^{(k)} \in \mathcal{L}^1_k$. Let $\epsilon>0$ be given. Then, by Lemma \ref{Density}, we can find $\mathcal{J}^{(k)}_1$, an integral operator with kernel $\mathcal{Z}^{(k)} \in W^{1,\infty}(\Lambda^k \times \Lambda^k)$ such that $\|\mathcal{J}^{(k)}-\mathcal{J}^{(k)}_1\|\leq \frac{\epsilon}{4}$. In particular, for all $s,t \in [0,T]$, it is the case that:

\begin{equation}
\notag
\big|Tr\, (\mathcal{J}^{(k)}-\mathcal{J}^{(k)}_1) (\widetilde{\gamma}^{(k)}_{\kappa;\,N,t}-\widetilde{\gamma}^{(k)}_{\kappa;\,N,s})\big| \leq \big|Tr\,(\mathcal{J}^{(k)}-\mathcal{J}^{(k)}_1) (\widetilde{\gamma}^{(k)}_{\kappa;\,N,t})\big|+\big|Tr\,(\mathcal{J}^{(k)}-\mathcal{J}^{(k)}_1) (\widetilde{\gamma}^{(k)}_{\kappa;\,N,s})|
\end{equation}
\begin{equation}
\notag
\leq \|\mathcal{J}^{(k)}-\mathcal{J}^{(k)}_1\| \cdot Tr\, \big|\widetilde{\gamma}^{(k)}_{\kappa;\,N,t}\big|+\|\mathcal{J}^{(k)}-\mathcal{J}^{(k)}_1\| \cdot Tr\, \big|\widetilde{\gamma}^{(k)}_{\kappa;\,N,s}\big| \leq \frac{\epsilon}{4}+\frac{\epsilon}{4}=\frac{\epsilon}{2}.
\end{equation}
In the last line, we used the duality pairing \eqref{duality_pairing} as well as \eqref{trace_gamma_tilde_equal_to_1}.

Furthermore, we notice that:
\begin{equation}
\notag
Tr \, \mathcal{J}^{(k)}_1 (\widetilde{\gamma}^{(k)}_{\kappa;\,N,t}-\widetilde{\gamma}^{(k)}_{\kappa;\,N,s})=\big \langle \mathcal{Z}^{(k)}, (\widetilde{\gamma}^{(k)}_{\kappa;\,N,t}-\widetilde{\gamma}^{(k)}_{\kappa;\,N,s}) \big \rangle_{L^2(\Lambda^k \times \Lambda^k)}.
\end{equation}
Hence, from \eqref{equicontinuity1}, it follows that there exists $\delta>0$ (which depends on $k$ and $\mathcal{J}^{(k)}_1$) such that for all $s,t \in [0,T]$ with $|s-t| \leq \delta$:
\begin{equation}
\notag
\big|Tr \, \mathcal{J}^{(k)}_1 (\widetilde{\gamma}^{(k)}_{\kappa;\,N,t}-\widetilde{\gamma}^{(k)}_{\kappa;\,N,s})\big| \leq \frac{\epsilon}{2}.
\end{equation}
In particular, by the triangle inequality, it follows that for all $s,t \in [0,T]$ with $|s-t| \leq \delta$:
\begin{equation}
\label{equicontinuity2}
\big|Tr \, \mathcal{J}^{(k)} (\widetilde{\gamma}^{(k)}_{\kappa;\,N,t}-\widetilde{\gamma}^{(k)}_{\kappa;\,N,s})\big| \leq \epsilon.
\end{equation}

\begin{remark}
Let us note that similar density arguments were used in order to obtain equicontinuity results in the periodic setting in \cite[Lemma 9.2]{ESY1} the non-periodic setting in \cite[Lemma 7.2]{ESY2}, \cite[Lemma 6.2]{ESY4}, \cite[Lemma 6.2]{ESY5}.
\end{remark}

Let us recall from \eqref{duality} that $\mathcal{L}^1_k= \mathcal{K}_k^{*}$. Since $\mathcal{K}_k$ is separable, there exists a countable dense subset of the unit ball of $\mathcal{K}_k$ given by $\{\mathcal{J}_{c,\ell}^{(k)},\,\ell \in \mathbb{N}\}.$ As in \cite{ESY2,ESY4,ESY5}, using the operators $\mathcal{J}_{c,\ell}^{(k)}$, it is possible to define a metric on $\mathcal{L}^1_k$ as:

\begin{equation}
\label{eta_k}
\eta_k(\gamma_1^{(k)},\gamma_2^{(k)}):=\sum_{\ell=1}^{+\infty}2^{-\ell} \cdot \big|Tr\,\mathcal{J}_{c,\ell}^{(k)}(\gamma_1^{(k)}-\gamma_2^{(k)})\big|.
\end{equation} 
The topology induced by the metric $\eta_k$ and the weak-$*$ topology are equivalent on the unit ball of $\mathcal{L}^1_k$, c.f. \cite[Theorem 3.16]{Rudin}. 

Let $C([0,T],\mathcal{L}^1_k)$ denote the space of all functions from $[0,T]$ which take values in $\mathcal{L}^1_k$ and which are continuous with respect to the topology given by the metric $\eta_k$ given in \eqref{eta_k}. On $C([0,T],\mathcal{L}^1_k)$, one defines the metric:
\begin{equation}
\label{etahat_k}
\widehat{\eta}_k(\gamma_1^{(k)}(\cdot),\gamma_2^{(k)}(\cdot)):=\mathop{sup}_{t \in [0,T]} \eta_k\,(\gamma_1^{(k)}(t),\gamma_2^{(k)}(t)).
\end{equation}
Finally, one defines $\tau_{prod}$ to be the product metric on $\bigoplus_{k \in \mathbb{N}} C([0,T],\mathcal{L}_1^k)$ obtained from the metric $\widehat{\eta}_k$ defined in \eqref{etahat_k}.

\subsection{Proof of Theorem \ref{Cubic_NLS_T3_derivation}}
\label{Proof of Theorem 2}

We will now use the above facts and prove Theorem \ref{Cubic_NLS_T3_derivation}:

\begin{proof}
It follows from \eqref{equicontinuity2} that the sequence $\big(\widetilde{\gamma}_{\kappa;\,N,t}^{(k)}\big)_N$ is equicontinuous with respect to $\eta_k$.
More precisely, given $\epsilon>0$, we find $M>1$ such that $\sum_{\ell>M} 2^{-\ell} <\frac{\epsilon}{4}.$ We note that $$\big|Tr \, \mathcal{J}_{c,\ell}^{(k)} (\widetilde{\gamma}^{(k)}_{\kappa;\,N,t}-\widetilde{\gamma}^{(k)}_{\kappa;\,N,s})\big| \leq \|\mathcal{J}_{c,\ell}^{(k)}\| \cdot  Tr \big| \widetilde{\gamma}^{(k)}_{\kappa;\,N,t}-\widetilde{\gamma}^{(k)}_{\kappa;\,N,s}\big| \leq 2.$$
Here, we used the triangle inequality and \eqref{trace_gamma_tilde_equal_to_1}.
Hence: 
\begin{equation}
\notag
\eta_k(\widetilde{\gamma}^{(k)}_{\kappa;\,N,t},\widetilde{\gamma}^{(k)}_{\kappa;\,N,s}) \leq \sum_{\ell=1}^{M} 2^{-\ell} \cdot \big|Tr\,\mathcal{J}_{c,\ell}^{(k)}(\widetilde{\gamma}^{(k)}_{\kappa;\,N,t}-\widetilde{\gamma}^{(k)}_{\kappa;\,N,s})\big| + \frac{\epsilon}{2},
\end{equation}
which can be made $\leq \epsilon$ by using \eqref{equicontinuity2} if we choose $|t-s|$ to be sufficiently small.
By the Arzel\`{a}-Ascoli theorem, it follows that the sequence $(\widetilde{\gamma}^{(k)}_{\kappa;\,N,t})_N$ is relatively compact with respect to $\widehat{\eta}_k$. By using an additional diagonal argument, it follows that $\big(\widetilde{\Gamma}_{\kappa;\,N,t}\big)_N$ is relatively compact with respect to $\tau_{prod}$. In particular, it follows that there exists $\overline{\Gamma}_{\kappa;\,\infty,t}=(\overline{\gamma}^{(k)}_{\kappa;\,\infty,t})_k \in \bigoplus_{k \in \mathbb{N}} C([0,T],\mathcal{L}_1^k)$ and $N_j \rightarrow \infty$ such that:
\begin{equation}
\label{Nj_infty2}
\widetilde{\Gamma}_{\kappa;\,N_j,t} \rightarrow \overline{\Gamma}_{\kappa;\,\infty,t} 
\end{equation}
as $j \rightarrow \infty$ with respect to $\tau_{prod}$. Let us note that $``\,\overline{\,\cdot\,}\,"$ in the above notation does not mean a complex conjugate. By taking subsequences, we can arrange for the $N_j$ in \eqref{Nj_infty1} and \eqref{Nj_infty2} to be the same. Let us note that the sequence $(N_j)_j$ depends on $\kappa$.

We now want to show that, for all $t \in [0,T]$:
\begin{equation}
\label{Gamma_tilde_bar}
\widetilde{\Gamma}_{\kappa;\,\infty,t}=\overline{\Gamma}_{\kappa;\,\infty,t}
\end{equation}
as elements of $L^2(\Lambda^k \times \Lambda^k)$.
Let us fix $k \in \mathbb{N}$ and $t \in [0,T]$.
Suppose that $F \in L^2(\Lambda^k \times \Lambda^k)$. The integral operator $T$ with kernel $F$ is Hilbert-Schmidt and hence is compact, i.e it belongs to $\mathcal{K}_k$. Since $\widetilde{\gamma}^{(k)}_{\kappa;\,N_j,t} \rightharpoonup^{*} \overline{\gamma}^{(k)}_{\kappa;\,\infty,t}$ in the weak-$*$ topology on $\mathcal{L}^1_k$ as $j \rightarrow \infty$, it follows from \eqref{duality_pairing} that:
\begin{equation}
\notag
Tr\, \big(T\, \widetilde{\gamma}^{(k)}_{\kappa;\,N_j,t} \big) \rightarrow Tr\, \big(T\, \overline{\gamma}^{(k)}_{\kappa;\,\infty,t} \big)
\end{equation}
as $j \rightarrow \infty$, which can be rewritten as:
\begin{equation}
\mathop{\int}_{\Lambda^k \times \Lambda^k} F(\vec{x}_k;\vec{x}'_k) \cdot \widetilde{\gamma}^{(k)}_{\kappa;\,N_j,t}(\vec{x}_k;\vec{x}'_k) \,d \vec{x}_k \, d\vec{x}'_k
\end{equation}
\begin{equation}
\notag
\rightarrow \mathop{\int}_{\Lambda^k \times \Lambda^k} F(\vec{x}_k;\vec{x}'_k) \cdot \overline{\gamma}^{(k)}_{\kappa;\,\infty,t}(\vec{x}_k;\vec{x}'_k) \,d \vec{x}_k \, d\vec{x}'_k.
\end{equation}
In particular:
\begin{equation}
\notag
\widetilde{\gamma}^{(k)}_{\kappa;\,N_j,t} 
\rightharpoonup
\overline{\gamma}^{(k)}_{\kappa;\,\infty,t}
\end{equation}
weakly in $L^2(\Lambda^k \times \Lambda^k)$ as $j \rightarrow \infty$.

Let us observe that \eqref{Nj_infty1} implies that:
\begin{equation}
\notag
\widetilde{\gamma}^{(k)}_{\kappa;\,N,t} \rightarrow \widetilde{\gamma}^{(k)}_{\kappa;\,\infty,t}
\end{equation}
strongly in $L^2(\Lambda^k \times \Lambda^k)$ as $j \rightarrow \infty$. In particular, $(\widetilde{\gamma}^{(k)}_{\kappa;\,N,t})_k$ also converges to $\widetilde{\gamma}^{(k)}_{\kappa;\,\infty,t}$ weakly in $L^2(\Lambda^k \times \Lambda^k)$.

By the uniqueness of weak limits, it follows that $\widetilde{\gamma}^{(k)}_{\kappa;\,\infty,t}=\overline{\gamma}^{(k)}_{\kappa;\,\infty,t}$ as elements of $L^2(\Lambda^k \times \Lambda^k)$. In particular, we can deduce \eqref{Gamma_tilde_bar}. We will henceforth write the limit as $\widetilde{\Gamma}_{\kappa;\,\infty,t}$.

To summarize, we know from Theorem \ref{EESY_Theorem} that $\widetilde{\Gamma}_{\kappa;\,\infty,t}$ is a mild solution of the Gross-Pitaevskii hierarchy which belongs to $L^{\infty}_{t \in [0,T]}\mathfrak{H}^1$.
Let us note that, by construction:
\begin{equation}
\notag
\widetilde{\gamma}^{(k)}_{\kappa;\,N,t}=Tr_{k+1,\ldots,N}\,\big|\widetilde{\Psi}_{N,t}^{\kappa} \rangle \langle \widetilde{\Psi}_{N,t}^{\kappa} \big| \geq 0
\end{equation}
when considered as an operator on $L^2(\Lambda^k)$. Furthermore, $\widetilde{\gamma}^{(k)}_{\kappa;\,N,t}$ is bounded and self-adjoint on $L^2(\Lambda^k)$.
We recall from \eqref{trace_gamma_equal_to_1} that $Tr\,\widetilde{\gamma}^{(k)}_{\kappa;\,N,t}=1$. Finally, we note that for all $k \in \mathbb{N}$ and for all $t \in [0,T]$, it is the case that $\widetilde{\gamma}^{(k)}_{\kappa;\,N_j,t} \rightharpoonup^{*}
\widetilde{\gamma}^{(k)}_{\kappa;\,\infty,t}
$ as $j \rightarrow \infty$ in the weak-$*$ topology on $\mathcal{L}^1_k$.
By \eqref{Approximation3}, we can deduce that, for all $k \in \mathbb{N}$, it is the case that $\widetilde{\gamma}^{(k)}_{\kappa;\,\infty,0}=|\phi \rangle \langle \phi|^{\otimes k}$.

In particular, it follows that $\widetilde{\Gamma}_{\kappa;\,\infty,t}$ is a mild solution of the Gross-Pitaevskii hierarchy with initial data given by $\widetilde{\Gamma}_{\kappa,\infty,0}=(|\phi \rangle \langle \phi|^{\otimes k})_k$, which satisfies the assumptions of Theorem \ref{Unconditional_uniqueness}. We can now apply Theorem \ref{Unconditional_uniqueness} and deduce that:

\begin{equation}
\label{Gamma_tilde_infty}
\widetilde{\Gamma}_{\kappa;\,\infty,t}=\big(|S_t(\phi) \rangle \langle S_t(\phi)|^{\otimes k}\big)_k.
\end{equation}

Consequently:

\begin{equation}
\label{weak_star_convergence1}
\widetilde{\gamma}^{(k)}_{\kappa;\,N_j,t} \rightharpoonup^{*} |S(t)\phi \rangle \langle S(t)\phi|^{\otimes k}
\end{equation}
as $j \rightarrow \infty$ in the weak-$*$ topology on $\mathcal{L}^1_k$.

Let us note that by \eqref{trace_gamma_tilde_equal_to_1} and the fact that $\|\phi\|_{L^2}=1$:
\begin{equation}
\label{weak_star_convergence_norm}
Tr \,\big|\widetilde{\gamma}^{(k)}_{\kappa;\,N_j,t}\big|=Tr\,\big||S(t)\phi \rangle \langle S(t) \phi|^{\otimes k}\big|=1.
\end{equation}

We observe that \eqref{weak_star_convergence1} and \eqref{weak_star_convergence_norm} imply that:
\begin{equation}
\label{gamma_tilde_strong_convergence}
Tr\,\Big|\widetilde{\gamma}^{(k)}_{\kappa;\,N_j,t} - |S(t)\phi \rangle \langle S(t)\phi|^{\otimes k} \Big|\rightarrow 0
\end{equation}
as $j \rightarrow \infty$. 
In other words, weak-$*$ convergence in $\mathcal{L}^1_k$ and convergence of the norms implies strong convergence in $\mathcal{L}^1_k$. This fact was already noted in \cite[Proposition 9.1 $iii)$]{ESY4}. It follows from a more general fact about convergence in the trace class proved by Arazy \cite{Arazy}. For related results, we refer the reader to \cite{Simon1}, \cite[Addendum H]{Simon2}, and \cite[Corollary 2.4]{LewinNamRougerie}, as well as \cite{Dell'Antonio} and \cite{Robinson}.
Let us note that, in \eqref{gamma_tilde_strong_convergence}, the sequence $(N_j)_j$ can be taken to be independent of $k \in \mathbb{N}$ and $t \in [0,T]$. We recall that $(N_j)_j$ does depend on $\kappa$.

Let us now prove an analogue of \eqref{gamma_tilde_strong_convergence} which is obtained by starting from the original  functions $\Psi_N$. In other words, we start from $(\Psi_N)_N \in \bigoplus_{N \in \mathbb{N}} L^2(\Lambda^N)$ which satisfies \eqref{Bounded_energy_per_particle} and \eqref{Asymptotic_factorization}, and we consider:
\begin{equation}
\notag
\gamma^{(k)}_{N,t}:=Tr_{k+1,\ldots,N} \,\big|\Psi_{N,t} \rangle \langle \Psi_{N,t}\big|
\end{equation}
as in \eqref{gammakNt_original}.

We want to estimate $Tr\,\big|\gamma^{(k)}_{N,t}-\widetilde{\gamma}^{(k)}_{\kappa;\,N,t}\big|$ for $\widetilde{\gamma}^{(k)}_{\kappa;\,N,t}$ as in \eqref{gamma_tilde_kappa_N_t}. Let us note that:

\begin{equation}
\notag
Tr\,\big|\gamma^{(k)}_{N,t}-\widetilde{\gamma}^{(k)}_{\kappa;\,N,t}\big|=Tr_{1,2\ldots,k}\,\Big|Tr_{k+1,\ldots,N} \big|\Psi_{N,t} \rangle \langle \Psi_{N,t}\big|-Tr_{k+1,\ldots,N}\,\big|\widetilde{\Psi}_{\kappa;\,N,t} \rangle \langle \widetilde{\Psi}_{\kappa;\,N,t}\big|\Big|
\end{equation}
\begin{equation}
\notag
\leq Tr\,\Big|\big|\Psi_{N,t} \rangle \langle \Psi_{N,t}\big|-\big|\widetilde{\Psi}_{\kappa;\,N,t} \rangle \langle \widetilde{\Psi}_{\kappa;\,N,t}\big|\Big|
\end{equation}
\begin{equation}
\notag
\leq Tr\,\Big|\big|\Psi_{N,t} \rangle \langle (\Psi_{N,t}-\widetilde{\Psi}_{\kappa;\,N,t})\big|\Big| + Tr\,\Big|\big|(\Psi_{N,t}-\widetilde{\Psi}_{\kappa;\,N,t}) \rangle \langle \Psi_{N,t} \big|\Big| + Tr\,\Big|\big|(\Psi_{N,t}-\widetilde{\Psi}_{\kappa;\,N,t}) \rangle \langle (\Psi_{N,t}-\widetilde{\Psi}_{\kappa;\,N,t})\big|\Big|
\end{equation}
\begin{equation}
\notag
\leq 2\|\Psi_{N,t}\|_{L^2(\Lambda^N)} \cdot \|\Psi_{N,t}-\widetilde{\Psi}_{\kappa;\,N,t}\|_{L^2(\Lambda^N)}+\|\Psi_{N,t}-\widetilde{\Psi}_{\kappa;\,N,t}\|_{L^2(\Lambda^N)}^2
\end{equation}
\begin{equation}
\label{gamma_gamma_tilde}
= 2\|\Psi_N\|_{L^2(\Lambda^N)} \cdot \|\Psi_N-\widetilde{\Psi}_{\kappa;\,N}\|_{L^2(\Lambda^N)}+\|\Psi_N-\widetilde{\Psi}_{\kappa;\,N}\|_{L^2(\Lambda^N)}^2.
\end{equation}
Above, we used the triangle inequality, the Cauchy-Schwarz inequality and the fact that $\|\Psi_{N,t}\|_{L^2(\Lambda^N)}$ and $\|\Psi_{N,t}-\widetilde{\Psi}_{\kappa;\,N,t}\|_{L^2(\Lambda^N)}$ are conserved in time. We note that \eqref{Asymptotic_factorization} implies that $\|\Psi_N\|_{L^2(\Lambda^N)} \rightarrow 1$ as $N \rightarrow \infty$. In particular, $(\|\Psi_N\|_{L^2(\Lambda^N)})_N$ is bounded. By using \eqref{Approximation2} and the fact that $\|\Psi_N\|_{L^2(\Lambda^N)} \lesssim 1$, it follows that the expression in \eqref{gamma_gamma_tilde} is:
\begin{equation}
\label{gamma_gamma_tilde2}
\leq C' \cdot \kappa^{\frac{1}{2}}
\end{equation}
for some $C'>0$, which is independent of $k,N$ and $\kappa$.

Let us recall that, the sequence $(N_j)_j$ for which \eqref{gamma_tilde_strong_convergence} holds, depends on $\kappa$. We will now emphasize this by writing $N_j=N_j(\kappa)$.
Consequently, by \eqref{gamma_gamma_tilde2}, we note that:

\begin{equation}
\notag
Tr\,\Big|\gamma^{(k)}_{N_j(\kappa),t} - |S(t)\phi \rangle \langle S(t)\phi|^{\otimes k} \Big| \leq Tr\,\Big|\widetilde{\gamma}^{(k)}_{\kappa;\,N_j(\kappa),t} - |S(t)\phi \rangle \langle S(t)\phi|^{\otimes k} \Big|+ Tr\, \Big|\gamma^{(k)}_{N_j(\kappa),t}-\widetilde{\gamma}^{(k)}_{\kappa;\,N_j(\kappa),t}\Big|
\end{equation}
\begin{equation}
\notag
\leq 
Tr\,\Big|\widetilde{\gamma}^{(k)}_{\kappa;\,N_j(\kappa),t} - |S(t)\phi \rangle \langle S(t)\phi|^{\otimes k} \Big|+C' \cdot \kappa^{\frac{1}{2}}
\end{equation}
We take a sequence $\kappa_n:=\frac{1}{n}$, and we apply \eqref{gamma_tilde_strong_convergence} and a diagonal argument in order to deduce that there exists a new sequence of positive integers tending to infinity, which is again denoted by $(N_j)_j$, which is independent of $k$ and $t$, and has the property that:
\begin{equation}
\label{gamma_strong_convergence}
Tr\,\Big| \gamma^{(k)}_{N_j,t}-|S(t)\phi \rangle \langle S(t)\phi|^{\otimes k}\Big| \rightarrow 0
\end{equation}
as $j \rightarrow \infty$. The theorem now follows from \eqref{gamma_strong_convergence}.
\end{proof}

\begin{remark}
\label{place_where_we_use_uniqueness}
We note that the place where we applied the uniqueness result from Theorem \ref{Unconditional_uniqueness} is in \eqref{Gamma_tilde_infty} above in order to deduce that $\widetilde{\Gamma}_{\kappa;\,\infty,t}$ is uniquely determined. In general, we can deduce that $\Gamma_{\infty,t}$ in Theorem \ref{EESY_Theorem} is uniquely determined provided that it satisfies the assumptions of Theorem \ref{Unconditional_uniqueness}. As we saw above, this was true in the intermediate step \eqref{Gamma_tilde_infty} of deriving the defocusing cubic nonlinear Schr\"{o}dinger equation on $\mathbb{T}^3$.
\end{remark}

\end{document}